\newcommand{\jo}[1]{}
\newcommand{\si}[1]{#1}

\RequirePackage{fix-cm}
\jo{
\documentclass[smallextended]{svjour3}       
}
\si{
\documentclass{article}
}
\jo{
\smartqed  
\journalname{Mathematical Programming}
\usepackage[misc,geometry]{ifsym} 
}

\si{
\usepackage{amsthm}
\usepackage[margin=1in]{geometry}
\geometry{letterpaper}
\theoremstyle{plain}

\newtheorem{theorem}{Theorem}
\newtheorem{lemma}{Lemma}
\newtheorem{corollary}{Corollary}
\newtheorem{definition}{Definition}

\newtheorem{remark}{Remark}

}
\usepackage{graphicx,fancyhdr,amsmath,amssymb,subfig,url}
\usepackage[hidelinks]{hyperref}
\usepackage{algorithm}
\usepackage{algpseudocode}
\usepackage{booktabs}
\usepackage{float}
\usepackage{multirow}
\usepackage{enumitem}
\usepackage[normalem]{ulem}
\usepackage{color}
\usepackage[numbers]{natbib}

\newcommand{\grad}{\nabla}

\def\Lag{\mathcal{L}}
\def\R{\mathbb{R}}
\newcommand{\maxIter}{300}
\newcommand{\vectorOfones}{e}

\newcommand{\minimize}{\mbox{minimize }}
\newcommand{\subjectTo}{\mbox{subject to }}

\newcommand{\figuresDir}[1]{#1} 

\begin{document}
\title{On the behavior of Lagrange multipliers in convex and nonconvex infeasible interior point methods\thanks{Gabriel Haeser was supported by the S\~ao Paulo Research Foundation (FAPESP grants 2013/05475-7 and 2016/02092-8) and the Brazilian National Council for Scientific and Technological Development (CNPq). Oliver Hinder was supported by the PACCAR INC stanford graduate fellowship.}
}

\jo{
\titlerunning{On the behavior of Lagrange multipliers in infeasible interior point methods}        
}

\jo{
\author{
Gabriel Haeser  \and
Oliver Hinder        \and
Yinyu Ye
}
}

\si{
\author{
Gabriel Haeser\thanks{
Department of Applied Mathematics, University of S\~ao Paulo, S\~ao Paulo SP, Brazil. Visiting Scholar at Department of Management Science and Engineering, Stanford University, Stanford CA 94305, USA. E-mail: ghaeser@ime.usp.br.}
\and
Oliver Hinder\thanks{Department of Management Science and Engineering, Stanford University, Stanford CA 94305, USA. E-mail: \{ohinder,yinyu-ye\}@stanford.edu}
\and
Yinyu Ye\footnotemark[3]
}
}

\jo{
\authorrunning{G. Haeser, O. Hinder, Y. Ye} 
\institute{G. Haeser  \at
Department of Applied Mathematics, 
\\Institute of
    Mathematics and Statistics, University of S\~ao Paulo, S\~ao Paulo SP,
    Brazil. Visiting Scholar at Department of Management Science and Engineering, Stanford University, Stanford CA 94305, USA.\\
 \email{ghaeser@ime.usp.br} 
 \and
              O. Hinder \Letter, Y. Ye \at
              Department of Management Science and Engineering, \\
              Stanford University, Stanford, CA 94305, USA.\\
              \email{\{ohinder,yinyu-ye\}@stanford.edu}     
}
}

\jo{
\date{Received: date / Accepted: date}
}
\si{
\date{\today}
}

\maketitle

\begin{abstract}
We analyze sequences generated by interior point methods (IPMs) in convex and nonconvex settings. We prove that moving the primal feasibility at the same rate as the barrier parameter $\mu$ ensures the Lagrange multiplier sequence remains bounded, provided the limit point of the primal sequence has a Lagrange multiplier. This result does not require constraint qualifications. We also guarantee the IPM finds a solution satisfying strict complementarity if one exists. On the other hand, if the primal feasibility is reduced too slowly, then the algorithm converges to a point of minimal complementarity; if the primal feasibility is reduced too quickly and the set of Lagrange multipliers is unbounded, then the norm of the Lagrange multiplier tends to infinity. 

Our theory has important implications for the design of IPMs. Specifically, we show that IPOPT, an algorithm that does not carefully control primal feasibility has practical issues with the dual multipliers values growing to unnecessarily large values. Conversely, the one-phase IPM of \citet*{hinder2018one}, an algorithm that controls primal feasibility as our theory suggests, has no such issue.
\jo{
}
\jo{
  
 }
\end{abstract}

 \section{Introduction}

This paper studies sequences generated by interior point methods (IPMs) that converge to Karush-Kuhn-Tucker (KKT) points of
\begin{subequations}\label{orginal-problem}
\begin{flalign}
\minimize& f(x)\\
\subjectTo &a(x)  + s = 0\\
&s \ge 0,
\end{flalign}
\end{subequations}
where the objective function $f:\R^n\to\R$ and the inequality constraints $a:\R^n\to\R^m$ are continuously differentiable functions. 

The central path generated by sequences of log barrier problems was introduced by \citet{mclinden1980analogue} for convex minimization subject to non-negativity constraints and generalized to linear inequalities by \citet{sonnevend1986analytical}. \citet{megiddo1989pathways} analyzed the path of primal-dual IPMs for linear programming and showed this path converges to a point satisfying strict complementarity. \citet{guler1993convergence} generalized this result to a large class of path-following IPMs for linear programming. Finding a strictly complementary solution is necessary to guarantee the super-linear convergence of IPMs for quadratic programs \cite[Proposition 5.1]{ye1993quadratic}. Furthermore, finding a strictly complementary solution for problems with nonconvex constraints ensures that the critical cone is reduced to a subspace. This subspace gives an efficient way to verify the second-order conditions by computing the least eigenvalue of the Hessian of the Lagrangian restricted to this subspace \cite[Theorem 4.4.2]{bazaraa2006practical}. In the nonlinear context, a strictly complementary solution may not always exist, but if it does, we would like to obtain it.

The results mentioned above implicitly avoid the issue of unbounded dual variables by starting from a strictly feasible point. However, this is rarely done in practice, as infeasible-start algorithms are often used \cite{lustig1990feasibility,mehrotra1992implementation}. \citet{mizuno1995surface} studies the sequences generated by these infeasible start algorithms for linear programming without assuming the existence of an interior point. They show that moving the constraint violation at the same rate as the barrier parameter $\mu$ guarantees that the dual multipliers are bounded. The boundedness of dual multipliers is practically important because the linear system solved at each iteration of an IPM can become poorly conditioned as the dual multipliers get large, making the linear system more difficult to solve, particularly using iterative methods \cite{gondzio2012matrix}. Some of our theoretical contributions can be viewed as extensions of this work to convex and nonconvex optimization. 

One alternative and elegant solution to these issues is the homogeneous algorithm \cite{andersen1998computational,andersen1999homogeneous,ye1994nl}. For convex problems, the homogeneous algorithm is guaranteed to produce a bounded sequence that converges to a maximally complementary solution. For linear programming, this guarantees that if the problem is feasible the algorithm will converge with bounded dual variables. However, it is unknown how to extend the homogeneous algorithm into nonconvex optimization. 
 
While many IPMs for general nonconvex optimization problems have been developed, there is little analysis of the sequences they generate. For example, it is unclear if IPMs can generate maximal complementarity solutions in the presence of nonconvexity. Furthermore, results showing that the sequence of dual iterates are bounded rely on the set of dual multipliers being bounded (which is equivalent to the Mangasarian-Fromovitz constraint qualification \cite{gauvin}). This assumption may be too restrictive because many practical optimization problems may lack a strict relative interior and therefore have an unbounded set of dual multipliers. For instance, we found that this is the case for $64$ out of the $95$ NETLIB problems (see Appendix \ref{app:experiment-details}).

Primal and dual sequences generated by nonconvex optimization algorithms such as IPMs, augmented Lagrangian methods and sequential quadratic programming have been analyzed in a number of works \cite{cpg,andreani2010new,andreani2011sequential,cakkt2,qiwei}. However, we are only aware of feasible IPMs being considered. Moreover, these studies have been focused on determining primal convergence to a KKT point, despite unboundedness of the dual sequence. Instead, we focus on guaranteeing boundedness and maximal complementarity of the dual sequence.

Next, we explain the current state of knowledge of primal and dual sequences generated by IPMs for linear programming. In particular, let $f(x):=g^\mathtt{T}x$, with constraints $a(x)+s=0, s\geq0$, where $a(x):=Mx-p$, $M$ is a matrix and $g, p$ are vectors. Many IPMs for linear programming compute direction $( d_{x}^{k},  d_{y}^{k}, d_{s}^k )$ at each iteration $k$ satisfying
\begin{subequations}\label{eq:lp-same-rate}
\begin{flalign}
M^{T} d_{y}^{k} &= -\eta^{k} ( g + M^{T} y^{k})  \\
M d_{x}^{k} + d_{s}^{k} &= -\eta^{k} (M x^{k} + s^{k} - p) \label{eq:primal-reduction-eta} \\
S^{k} d_{y}^k + Y^{k} d_{s}^{k} + S^{k} y^{k} &=(1 - \eta^{k}) \mu^{k} \vectorOfones,
\end{flalign}
\end{subequations}
where $Y^k$ and $S^k$ are the diagonal matrices defined by $y^k$ and $s^k$, and $\vectorOfones$ is a vector of ones. The values $\mu^{k+1} :=  (1 - \eta^{k}) \mu^{k}$ and $\eta^{k} \in (0,1)$ are chosen, for example, using a predictor-corrector technique \cite{mehrotra1992implementation}, see also \cite[Algorithm 14.3]{nocedal2006numerical}.  The iterates are updated according to $(x^{k+1},y^{k+1},s^{k+1}) \gets (x^k,y^k,s^{k}) + \alpha^k (d_{x}^k,d_{y}^k,d_{s}^k)$, where $\alpha^k \in (0,1]$ is the step size. Methods that choose their iterates in this way reduce the primal feasibility and complementarity at approximately the same rate \cite{lustig1994interior,mehrotra1992implementation,ye1994nl}, which we formalize as follows.
Suppose the IPM converges to an optimal solution as $\mu^{k} \rightarrow 0$. Then 
a subsequence of iterates satisfy $x^{k} \rightarrow x^{*}$, $s^{k} \rightarrow s^{*}$, and:
\begin{subequations}\label{seq-ipm-eq:nice-IPM}
\begin{flalign}
a(x^{*}) + s^{*} &= 0 \label{seq-ipm-eq:final-feasibility}\\
b \mu^{k} \leq s_i^{k} y_i^{k}&\leq c \mu^{k}  \quad \text{for all $i$} \label{seq-ipm-eq:comp-slack} \\
 \ell \mu^{k} \leq a_i(x^{k}) + s_i^{k}  &\leq u \mu^{k} \quad \text{for all $i$} \label{seq-ipm-eq:primal-rate}    \\
\| \grad_x \Lag(x^k,y^k) \| &\leq d \mu^k(\| y^{k} \|_1 + 1) \label{seq-ipm-eq:dual-feas} \\
s^{k}, y^{k} &\ge 0, \label{seq-ipm-eq:dual-feas2}
\end{flalign}
\end{subequations}
where $\mu^{k}>0$ is the barrier parameter, $0<b\leq c$, $0<\ell\leq u$, $d\geq0$ are real constants independent of $k$, $\| \cdot  \|_{p}$ denotes the $\ell_{p}$ norm, $\| \cdot  \|$ the Euclidean norm, and the Lagrangian $\Lag : \R^n\times\R^m \rightarrow \R$ is
\begin{flalign}
\Lag(x, y) := f(x) + y^{T} a(x).
\end{flalign}
Inequality~\eqref{seq-ipm-eq:comp-slack} ensures perturbed complementarity approximately holds. Inequality~\eqref{seq-ipm-eq:primal-rate} guarantees that primal feasibility is reduced at the same rate as complementarity. Inequality~\eqref{seq-ipm-eq:dual-feas} ensures that scaled dual feasibility is reduced fast enough. The set of inequalities~\eqref{seq-ipm-eq:nice-IPM} has a natural interpretation as a `shifted log barrier': a sequence of approximate KKT points to the problem,
\begin{flalign*}
\minimize ~ f(x) - \mu^k \sum_{i=1}^{m} \log(\mu^k r_i - a_i(x)),
\end{flalign*}
with the vector $r$ satisfying $\ell \le r_i \le u$.

The main contribution of \citet*{mizuno1995surface} was to show that IPMs for linear programming have bounded Lagrange multiplier sequences and satisfy strict complementarity when \eqref{seq-ipm-eq:nice-IPM} holds. \citet*{hinder2018one} show it is also possible to develop IPMs that satisfy \eqref{seq-ipm-eq:nice-IPM} even if $f$ and $a$ are nonlinear. In particular, they give an IPM where, if the primal variables are bounded and the algorithm does not return a certificate of local primal infeasibility, a subsequence of the iterates satisfy \eqref{seq-ipm-eq:nice-IPM}. This motivates us to show given a sequence satisfying \eqref{seq-ipm-eq:nice-IPM}, even if the objective and constraints are  \emph{nonlinear} the dual multipliers are still, under general conditions, well-behaved.

\subsection{Summary of contributions}

Now, assuming conditions \eqref{seq-ipm-eq:nice-IPM} and that the problem is convex (or certain sufficient conditions for local optimality hold) we show:

\begin{enumerate}[label=(\alph*)]
\item If there exists a Lagrange multiplier at the point $x^{*}$, then the sequence of Lagrange multipliers approximations $\{ y^{k} \}$ is bounded (see Theorem~\ref{convexbound} and Theorem~\ref{thm:nonconvex-bound-lag} for the convex and nonconvex case respectively).
\item If $y^{k} \rightarrow y^{*}$, then among the set of Lagrange multipliers at the point $x^{*}$, the point $y^{*}$ is maximally complementary (see Theorems~\ref{thm:maxcomp} and \ref{thm:maxcomp-suf}).
\end{enumerate}

Consider the case that \eqref{seq-ipm-eq:primal-rate} does not hold, i.e., the primal feasibility is not being reduced at the same rate as complementarity. We argue that this is poor algorithm design, because if problem \eqref{orginal-problem} is convex then:

\begin{enumerate}[label=(\alph*)]
\item\label{list:issue-1} If we reduce the primal feasibility faster than the barrier parameter $\mu^{k}$ and the set of dual multipliers at the point $x^{*}$ is unbounded, then $\| y^{k} \| \rightarrow \infty$ (see Theorem~\ref{thm:divergence1}).
\item If we reduce the primal feasibility slower than the barrier parameter $\mu^{k}$ and $y^{k} \rightarrow y^{*}$, then $y^{*}$ is a minimally complementary Lagrange multiplier associated with $x^{*}$ (see Theorem~\ref{mincomp}).
\end{enumerate}

Our central claim is that many implemented interior point methods, especially for nonlinear optimization, such as IPOPT \cite{wachter2006implementation}, suffer from the problems described above because they fail to control the rate at which they reduce primal feasibility (specifically IPOPT suffers from deficiency~\ref{list:issue-1}). For linear programs, these methods solve systems of the form \cite[equation (9)]{wachter2006implementation}
\begin{subequations}
\begin{flalign}
M^{T} d_{y}^{k} &= -\left( g + M^{T} y^{k} \right)  \\
M d_{x}^{k} + d_{s}^{k} &= - \left( M x^{k} + s^{k} - p \right) \label{eq:agg-primal} \\
S^{k} d_{y}^{k} + Y^{k} d_{s}^{k} + S^{k} y^{k} &= \mu^{k} \vectorOfones,
\end{flalign}
\end{subequations}
where the notation follows \eqref{eq:lp-same-rate}. Equation~\eqref{eq:agg-primal} aims to reduce the constraint violation to zero at each iteration. Contrast \eqref{eq:agg-primal} with equation~\eqref{eq:primal-reduction-eta} that aims to reduce the constraint violation by $\eta^{k}$, the same amount by which complementarity is reduced. As we demonstrate in Section~\ref{sec:experiments}, a consequence of the implementation choices in IPOPT, primal feasibility is usually reduced faster than complementarity. Therefore, as our theory suggests, these IPMs have issues with the Lagrange multipliers sequence diverging. 

IPOPT attempts to circumvent this issue by perturbing the original constraint $a(x)\leq0$ to create an artificial interior as follows:
\begin{flalign}
\label{ipopt-pert}
a(x) \le \delta \vectorOfones,
\end{flalign}
for some $\delta > 0$ (see Section~3.5 of \cite{wachter2006implementation}). While this technically solves the issue as the theoretical assumptions of \cite{wachter2005line} are now met, it is not an elegant solution and causes undesirable behavior. For example, we show in Section~\ref{sec:experiments} that the dual variable may still {\it spike} before converging. Furthermore, if $\delta$ is selected to be large, the constraints will be only loosely satisfied at the final solution.

We proceed as follows. Section~\ref{sec:simple-example} gives a simple example illustrating the phenomena studied. Section~\ref{sec:lag-bounded} shows that reducing the primal feasibility at the same rate as complementarity ensures the dual multiplier sequence remains bounded and satisfies maximal complementarity. Section~\ref{negative} explains that reducing the constraint violation too quickly causes the dual multiplier sequence to be unbounded, while reducing it too fast causes the them to tend towards a minimal complementarity solution. Section~\ref{sec:experiments} shows empirically how strategies that reduce the constraint violation too fast, such as the one employed by IPOPT, can have issues with extremely large dual multipliers. Section~\ref{conclusions} presents our final remarks.

\subsection{A simple example demonstrating phenomena}\label{sec:simple-example}

Consider the following simple linear programming problem:
\begin{subequations}\label{eq:simple-lp-example}
\begin{flalign}
\minimize& 0\\
\subjectTo& x \le 1\\
& x \ge 1.
\end{flalign}
\end{subequations}
By adding a feasibility perturbation $\delta > 0$ and a log barrier term $\mu\geq0$, we get
\begin{subequations}\label{eq:simple-lp-example2}
\begin{flalign}
\minimize& - \mu \log{(x - 1 + \delta)} - \mu \log{(1 - x + \delta)} \\
\subjectTo& x \ge 1 - \delta \\
& x \le 1 + \delta. \label{eq:simple-lp-example:con2} 
\end{flalign}
\end{subequations}
The associated KKT system is
\begin{subequations}\label{eq:simple-lp-example-ktt}
\begin{flalign}
& y_{1} - y_{2} = 0 \\
& x + s_{1} = 1 + \delta \\
& x - s_{2} = 1 - \delta \\
& s_{1} y_{1} = \mu \quad s_{2} y_{2} = \mu \\
& y_{1}, y_{2} \geq 0 \quad s_{1}, s_{2} \geq 0.
\end{flalign}
\end{subequations}

Observe that the orignal problem (corresponding to $\delta = \mu = 0$) has a unique optimal primal solution at $x^{*} := 1$ and $s^{*}:=(0,0)$, with dual solutions $y_{1}^{*} = y_{2}^{*}$ for any $ y_{2}^{*} \ge 0$. Therefore the set of dual variables is unbounded. However, for any $\delta, \mu > 0$, the solution to system \eqref{eq:simple-lp-example-ktt} is
\begin{flalign*}\label{ex:lagr}
x = 1 \quad s_{1} = \delta \quad s_{2} = \delta \quad  y_{1} = \frac{\mu}{\delta} \quad y_{2} = \frac{\mu}{\delta}. 
\end{flalign*}
From these equations, we can see that if $\delta$ and $\mu$ move at the same rate, then both strict complementarity and boundedness of the dual variables will be achieved. But if $\delta$ reduces faster than $\mu$, i.e., $\delta/\mu \rightarrow 0^+$, then the dual variables sequence is unbounded. Alternatively, if $\delta$ moves slower than $\mu$, i.e., $\delta/\mu \rightarrow \infty$, then strict complementarity will not hold.

Now, if $\delta>0$ is fixed at a small value as in the IPOPT strategy \eqref{ipopt-pert}, the dual sequence will initially grow very fast before stabilizing when the barrier parameter $\mu$ is sufficiently reduced. We confirm this hypothesis by solving the linear programming problem \eqref{eq:simple-lp-example} with perturbations $\delta>0$ using IPOPT, and we compare it with a {\it well-behaved IPM} \cite{hinder2018one} that moves complementarity at the same rate as primal feasibility, that is, satisfies \eqref{seq-ipm-eq:final-feasibility}--\eqref{seq-ipm-eq:dual-feas2}. For this experiment, we turn off IPOPT's native perturbation strategy \eqref{ipopt-pert}. In Figure~\ref{graph:toy-example} we plot the maximum dual variables at each iteration, given by the two methods for different  perturbation sizes. While perturbing the linear program prevents the dual variables of IPOPT from increasing indefinitely, the dual variables still {\it spike}. For example, with $\delta = 10^{-8}$ the maximum dual variable of IPOPT peaks at $10^4$ on iteration 4 before sharply dropping on the next iteration to $9$. Picking a smaller $\delta$, e.g., $\delta = 10^{-5}$, ensures a smaller peak at the cost of solving the problem to a lower accuracy. On the other hand, the maximum dual variable for the well-behaved IPM remains below $1.5$ irrespective of the perturbation size. Furthermore, with $\delta = 0.0$ the curve for the well-behaved IPM is essentially flat. 


\begin{figure}[H]
\centering
\includegraphics[scale=0.6]{\figuresDir{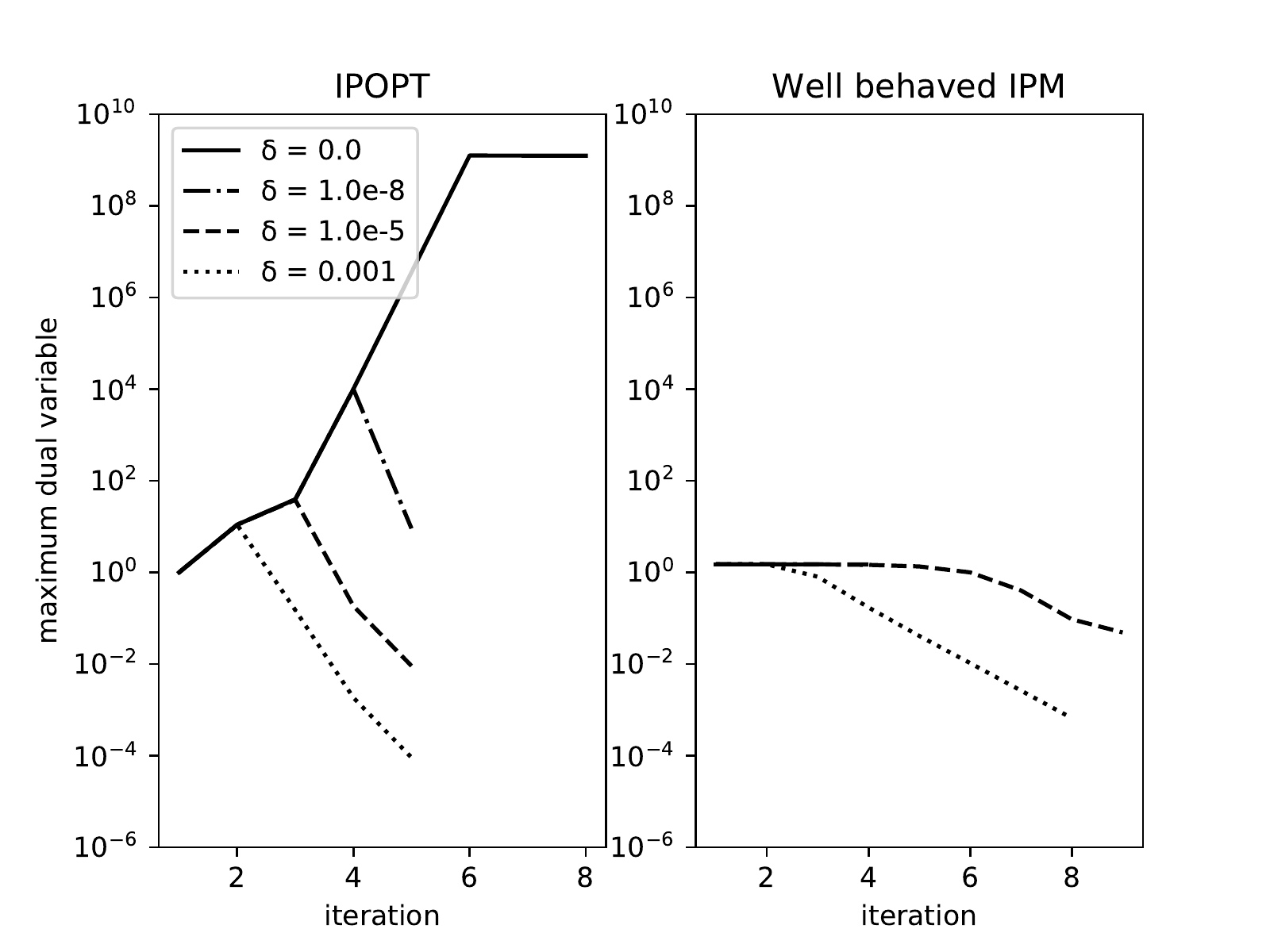}}
\caption{Comparison of the maximum dual variable value (vertical axis) against iterations (horizontal axis) using IPOPT and a well-behaved IPM \cite{hinder2018one} as the perturbation $\delta$ is changed.}\label{graph:toy-example}
\end{figure}


More thorough numerical experiments are given in Section~\ref{sec:experiments}, but first we establish our general theory. 

\section{Boundedness and maximal complementarity}
\label{sec:lag-bounded}


In this section, we show that when feasibility is reduced at the same rate as complementarity, the dual variables are bounded and satisfy maximal complementarity. But first we establish some basic results for convex problems on the optimality of limit points of sequences satisfying \eqref{seq-ipm-eq:nice-IPM}.

\paragraph{Notation.}
When it is clear from the context, we omit a quantifier ``$\forall k$'' when stating properties of every sufficiently large element of a sequence indexed by $k=1,2,\dots,\infty$. The Euclidean norm is denoted by $\|\cdot\|$; otherwise the $\ell_p$-norm (we only use $p=1, \infty$) is denoted by $\|\cdot\|_p$.

The following lemma gives a sufficient sequential condition for global optimality in the convex case. In our setting, the lemma is slightly more general than results found in the literature, e.g., see \cite[Corollary 3.1]{jeya2003}, \cite[Theorem 4.2]{andreani2010new}, \cite[Theorem 2.2]{haeser2011}, and \cite[Theorem 3.2]{giorgi2016}. Our condition is in fact equivalent to the one from \cite{giorgi2016} but with a redundant assumption omitted.

\begin{lemma}
\label{liminf}
If $f$ and $a_i$ for $i=1,\dots,m$ are convex functions, and $\{(x^k,y^k)\}\subset\mathbb{R}^n\times\mathbb{R}^m$ are such that
\begin{enumerate}
\item $x^k\to x^*$ with $a(x^*)\leq0$,
\item $y^k\geq0$,
\item $\lim\inf a(x^k)^\mathtt{T}y^k\geq0$,
\item $\grad_x\Lag(x^k,y^k)\to0$,
\end{enumerate}
then, $x^*$ is a solution of \eqref{orginal-problem}.
\end{lemma}
\begin{proof}
Given $x$ with $a(x)\leq0$, we have $$f(x)\geq \Lag(x,y^k)\geq \Lag(x^k,y^k)+\grad_x \Lag(x^k,y^k)^\mathtt{T}(x-x^k).$$ Hence, \begin{equation}\label{aux1}a(x^k)^\mathtt{T}y^k\leq f(x)-f(x^k)+\grad_x \Lag(x^k,y^k)^\mathtt{T}(x^k-x).\end{equation} Thus, for $x=x^*$, we have $\lim\sup a(x^k)^\mathtt{T}y^k\leq0$. The assumption gives $a(x^k)^\mathtt{T}y^k\to0$. Taking the limit in (\ref{aux1}) we have $f(x)\geq f(x^*)$ and the result follows.\jo{\qed}
\end{proof}

The following lemma gives a sufficient condition for verifying the conditions of Lemma \ref{liminf} under our slack variable formulation, which suits better our interior point framework.

\begin{lemma}
\label{slack-lemma}
If $f$ and $a_i$ for $i=1,\dots,m$ are convex functions, and $\{(x^k,y^k,s^k)\}\subset\mathbb{R}^n\times\mathbb{R}^m\times\mathbb{R}^m$ are such that
\begin{enumerate}
\item $x^k\to x^*$ with $a(x^*)\leq0$ and $s^k\to-a(x^*)$,
\item $y^k\geq0$ and $s^k\geq0$,
\item $(y^k)^\mathtt{T}s^k\to0$,
\item $a_i(x^k)+s_i^k\geq0$ for all $i: a_i(x^*)=0$,
\item $\grad_x \Lag(x^k,y^k)\to0$,
\end{enumerate}
then, $x^*$ is a solution of \eqref{orginal-problem}.
\end{lemma}
\begin{proof}
For $i: a_i(x^*)=0$, we have $a_i(x^k)y_i^k\geq-s_i^ky_i^k\to0$, while if $a_i(x^*)<0$, we have $y_i^k\to0$. The result follows from Lemma \ref{liminf}.\jo{\qed}
\end{proof}

We note that even in the nonconvex case, the existence of sequences satisfying the conditions of Lemmas \ref{liminf} and \ref{slack-lemma} are necessary at a local solution $x^*$, without constraint qualifications. This follows from the necessary existence of sequences $x^k\to x^*, y^k\geq0$ with $\grad_x\Lag(x^k,y^k)\to0, a_i(x^k)y_i^k\to0$ for all $i$, when $x^*$ is a local solution, given in \cite[Theorem 3.3]{andreani2010new}, by defining $s_i^k:=\max\{0,-a_i(x^k)\}$ for all $i$ and all $k$. See also \cite{cakkt2}.

\subsection{Boundedness of the dual sequence}

The boundedness of the dual sequence is an important property because the algorithm is otherwise prone to numerical instabilities.

In Theorem~\ref{convexbound}, we consider problems involving convex functions where the algorithm is converging to a KKT point. We show that if the primal feasibility, (scaled) dual feasibility and complementarity converge at the same rate, then the dual sequence $\{y^k\}$ is bounded. We refer to \cite[Theorem 4]{mizuno1995surface} for a more general result when the functions $f$ and $a$ are linear. This result is extended in Theorem~\ref{thm:nonconvex-bound-lag} to situations where the optimization problem may involve nonconvex functions. 

We try to present as few assumptions as possible; for example, assumptions are often placed only on constraints that are active at the limit. However, since in practice the active constraints are unknown, we advocate using IPMs that satisfy \eqref{seq-ipm-eq:nice-IPM}. All assumptions on the sequence of iterates made on theorems in this section can be subsumed by \eqref{seq-ipm-eq:nice-IPM} ignoring constant factors.

\begin{theorem}
\label{convexbound}
Suppose $f$ and $a_i$ for $i=1,\dots, m$ are convex functions and $\{(x^k,y^k,s^k,\mu^k)\}\subset\mathbb{R}^n\times\mathbb{R}^m\times\mathbb{R}^m\times\R$ with $\mu^{k}>0$ for all $k$ and $\mu^{k}\to0$ are such that:
\begin{enumerate}
\item $x^k\to x^*$  with $a(x^*)\leq0$ and $s^k\to-a(x^*)$,
\item $y^k\geq0$ and $s^k\geq0$,
\item for some $c\geq0$, $(y^k)^\mathtt{T}s^k\leq\mu^{k}c$,
\item for some $0<\ell\leq u$, $\mu^{k}\ell\leq a_i(x^k)+s_i^k\leq\mu^{k} u$ for all $i: a_i(x^*)=0$,
\item for some $d\geq0$, $\|\grad_x \Lag(x^k,y^k)\|\leq d\mu^{k}(\|y^k\|_1+1)$.
\end{enumerate}
Then, $x^*$ is a solution of \eqref{orginal-problem}. If $x^*$ is a KKT point, then $$\lim\sup\|y^k\|_1\leq\frac{2u}{\ell}\|y^*\|_1+\frac{4c}{\ell}m+\frac{2(c+d)}{\ell},$$ where $y^*$ is any Lagrange multiplier associated with $x^*$, i.e., $\grad \Lag(x^{*},y^{*}) = 0$, $y^{*} \ge 0$, and $a(x^{*})^T y^{*} = 0$. 
\end{theorem}

\begin{proof}
We have by convexity of $\Lag(x, y^{k})$ in $x$ that
$$f(x^*)\geq \Lag(x^*,y^k)\geq \Lag(x^k,y^k)+\grad_x \Lag(x^k,y^k)^\mathtt{T}(x^*-x^k),$$ which gives
\begin{equation}
\label{referee1}
f(x^*)-f(x^k)\geq a(x^k)^\mathtt{T}y^k+\grad_x \Lag(x^k,y^k)^\mathtt{T}(x^*-x^k).
\end{equation}
Also, $$a(x^k)^\mathtt{T}y^k=(a(x^k)+s^k)^\mathtt{T}y^k-(s^k)^\mathtt{T}y^k\geq\sum_{i:a_i(x^*)=0}\mu^{k}\ell y_i^k+\sum_{i:a_i(x^*)<0}(a_i(x^k)+s_i^k)y_i^k-\mu^{k}c.$$ Since $s_i^ky_i^k\geq0$ and $a_i(x^k)y_i^k\geq\frac{a_i(x^k)}{s_i^k}\mu^{k}c\geq-2\mu^{k}c$ for $i:a_i(x^*)<0$ and sufficiently large $k$, we have 
\begin{eqnarray}
a(x^k)^\mathtt{T}y^k\geq\ell\mu^{k}\sum_{i:a_i(x^*)=0}y_i^k-\sum_{i:a_i(x^*)<0}2c\mu^{k}-c\mu^{k}=\nonumber\\
\ell\mu^{k}\|y^k\|_1-\sum_{i:a_i(x^*)<0}(2c+\ell y_i^k)\mu^{k}-c\mu^{k}.\label{akyk-bound}
\end{eqnarray}
Also, $\grad_x \Lag(x^k,y^k)^\mathtt{T}(x^*-x^k)\geq-d\mu^{k}(\|y^k\|_1+1)\|x^*-x^k\|$.
Hence, substituting this and \eqref{akyk-bound} back in \eqref{referee1} we get$$f(x^*)-f(x^k)\geq\ell\mu^{k}\|y^k\|_1-d\mu^{k}(\|y^k\|_1+1)\|x^*-x^k\|-\sum_{i:a_i(x^*)<0}(2c+\ell y_i^k)\mu^{k}-c\mu^{k}.$$
We can take $k$ large enough that $\ell\mu^{k}\|y^k\|_1-d\mu^{k}(\|y^k\|_1+1)\|x^*-x^k\|\geq\frac{\ell}{2}\mu^{k}\|y^k\|_1-d\mu^{k}$, so that, 
\begin{equation}
\label{referee2}
f(x^*)-f(x^k)\geq\frac{\ell}{2}\mu^{k}\|y^k\|_1-\sum_{i:a_i(x^*)<0}(2c+\ell y_i^k)\mu^{k}-(c+d)\mu^{k}.
\end{equation}
Since $\ell>0$ and $y_i^k\to0$ for $i:a_i(x^*)<0$, we have by \eqref{referee2} that $\mu^{k}\|y^k\|_1\to0$. This implies that $\grad_x\Lag(x^k,y^k)\to0$, and we can use Lemma \ref{slack-lemma} to conclude that $x^*$ is a solution.

On the other hand, let $y^*\in\mathbb{R}^m$ be a Lagrange multiplier associated with $x^*$. Then, $f(x^*)=\Lag(x^*,y^*)\leq \Lag(x^k,y^*)$, which, combining with \eqref{referee2} yields 
\begin{equation}
\label{referee3}
\frac{\ell}{2}\mu^{k}\|y^k\|_1-\sum_{i:a_i(x^*)<0}(2c+\ell y_i^k)\mu^{k}-(c+d)\mu^{k}\leq f(x^*)-f(x^k)\leq a(x^k)^\mathtt{T}y^*.
\end{equation}
But $a(x^k)^\mathtt{T}y^*=(a(x^k)+s^k)^\mathtt{T}y^*-(s^k)^\mathtt{T}y^*\leq\mu^{k}u\|y^*\|_1$,
which implies, by dividing \eqref{referee3} by $\frac{\ell}{2}\mu^k$, that $$\|y^k\|_1\leq\frac{2u}{\ell}\|y^*\|_1+\sum_{i:a_i(x^*)<0}\left(\frac{4c}{\ell}+2y_i^k\right)+\frac{2(c+d)}{\ell}.$$ Since $y_i^k\to0$ for $i: a_i(x^*)<0$, the result follows.\jo{\qed}
\end{proof}

Optimization problems with complementarity constraints are an important class of nonconvex optimization problems where the Mangasarian-Fromovitz constraint qualification fails.  Typically, specialized IPMs for these problems are developed \cite{leyffer2006interior,benson2002interior}. The following corollary focuses on convex programs with complementarity constraints. It shows that any general purpose IPM satisfying \eqref{seq-ipm-eq:nice-IPM} has a bounded dual multipliers sequence under general conditions. 

\begin{corollary}
\label{coro:mpec-bound-lag}
Let $\{(x^k,y^k,s^k,\mu^k)\}\subset\mathbb{R}^n\times\mathbb{R}^m\times\mathbb{R}^m\times\R$ with $\mu^{k}>0$ and $\mu^{k}\to0$ be such that assumptions 1-5 of Theorem~\ref{convexbound} hold. Assume that problem \eqref{orginal-problem} is a convex program with complementarity constraints, that is:
\begin{subequations}\label{comp-problem}
\begin{flalign}
\minimize & f(x)   \\
\subjectTo & g(x) \le 0 \\
& x_i x_j \le 0 \quad (i,j) \in \mathbf{C} \label{comp-problem-comp-cons} \\
& x \ge 0,
\end{flalign}
\end{subequations}
where $\mathbf{C} \subseteq \{1, \dots, n \} \times \{1, \dots, n \}$ and $g_i : \R^{n} \rightarrow \R$ is convex for $i = 1, \dots, p$ ($p \le m$). Assume $x_i^{*} + x_j^{*} > 0$ for all $(i,j) \in \mathbf{C}$. Under these assumptions, $x^{*}$ is a local minimizer. Furthermore, if $x^{*}$ is a KKT point then $\{y^k\}$ is bounded.
\end{corollary}
\begin{proof}
To prove this result, it is sufficient to show that we are implicitly generating a sequence satisfying the assumptions of Theorem~\ref{convexbound}, where \eqref{orginal-problem} is replaced by the convex program
\begin{subequations}\label{convexified-mpec}
\begin{flalign}
\minimize & f(x)   \\
\subjectTo & g(x) \le 0 \\
& x_i^{*} x_j \le 0  \quad (i,j) \in \mathbf{F} \label{convexified-mpec-comp-cons} \\
& x \ge 0,
\end{flalign}
\end{subequations}
where $\mathbf{F} = \{ (l,k) \in \mathbf{C} : x_l^{*} > 0  \} \cup \{ (k,l): (l,k) \in \mathbf{C}, ~x_k^{*} > 0  \}$. By the strict complementarity assumption, we deduce that if $(i,j) \in \mathbf{C}$ then either $(i,j) \in \mathbf{F}$ or $(j,i) \in \mathbf{F}$; i.e., there is a one-to-one correspondence between constraints in \eqref{comp-problem} and \eqref{convexified-mpec}. Therefore if $\tilde{x}^{*} \in \{ x : \| \tilde{x}^{*} - x^{*} \|_{\infty} \le \frac{1}{2} \min_{(i,j) \in F} x^{*}_i \}$ is feasible for \eqref{comp-problem}, then $\tilde{x}^{*}$ is feasible for \eqref{convexified-mpec}. We deduce that any minimizer for \eqref{convexified-mpec} is a local minimizer for \eqref{comp-problem}. Now, $\frac{| x_i^k x_j^k - x_i^{*} x_j^k |}{\mu^k} = | x_i^k - x_i^{*} | \frac{x_j^k}{\mu^k} \le  | x_i^k - x_i^{*} | \frac{u}{x_i^k} \rightarrow 0$ for all $(i,j) \in \mathbf{F}$. Hence, the sequence $(x^k,y^k,s^k,\mu^k)$ satisfies the assumptions of Theorem~\ref{convexbound}, where \eqref{orginal-problem} is replaced by \eqref{convexified-mpec}. \jo{\qed}
\end{proof}

We now present a nonconvex version of Theorem \ref{convexbound}. For this, we assume that the limit point $x^*$ satisfies a sufficient optimality condition based on the star-convexity concept described below. This definition is a local version of the one from \cite{nesterov2006cubic}.

\begin{definition}
Let a function $q:\R^n\to\R$, a point $x^*\in\R^n$, and a set $S\subseteq\R^n$ be given.  We say that $q$ is star-convex around $x^*$ on $S$ when $$q(\alpha x+(1-\alpha)x^*)\leq\alpha q(x)+(1-\alpha)q(x^*)\mbox{ for all }\alpha\in[0,1]\mbox{ and }x\in S.$$
\end{definition}

\begin{theorem}
\label{thm:nonconvex-bound-lag}
Let $\{(x^k,y^k,s^k,\mu^k)\}\subset\mathbb{R}^n\times\mathbb{R}^m\times\mathbb{R}^m\times\R$ with $\mu^{k}>0$ and $\mu^{k}\to0$ be such that
\begin{enumerate}
\item $x^k\to x^*$  with $a(x^*)\leq0$ and $s^k\to-a(x^*)$,
\item $y^k\geq0$ and $s^k\geq0$,
\item for some $c\geq0$, $(y^k)^\mathtt{T}s^k\leq\mu^{k}c$,
\item for some $0<\ell\leq u$, $\mu^{k}\ell\leq a_i(x^k)+s_i^k\leq\mu^{k} u$ for all $i: a_i(x^*)=0$,
\item for some $d\geq0$, $\|\grad_x \Lag(x^k,y^k)\|\leq d\mu^{k}(\|y^k\|_1+1)$,
\item $x^*$ is a KKT point with Lagrange multiplier $y^*$,
\item There exist $\theta \ge 0$ and a neighborhood $\mathcal{B}$ of $x^*$ such that $\hat{\Lag}_{k}(x):=\Lag(x,y^k)+\theta a(x)^\mathtt{T}Y^{k}a(x)$ for all $k$ and $\hat{\Lag}_*(x):=\Lag(x,y^*)+\theta a(x)^\mathtt{T}Y^*a(x)$ are star-convex around $x^*$ on $\mathcal{B}$, where $Y^{k}=diag(y^k)$ and $Y^*=diag(y^*)$.\label{thm:nonconvex-bound-lag:final-ass}
\end{enumerate}
Then, $\{y^k\}$ is bounded.
\end{theorem}

\begin{proof}
From the definition of star-convexity of $\hat\Lag_{k}$, taking limit in $\alpha$, we have
\begin{eqnarray*}
f(x^*)+ \theta a(x^*)^\mathtt{T}Y^{k}a(x^*)\geq \hat \Lag_{k}(x^*)\geq \hat \Lag_{k}(x^k)+\grad_x \hat \Lag_{k}(x^k)^\mathtt{T}(x^*-x^k)=\\
\Lag(x^k,y^k)+\grad_x \Lag(x^k,y^k)^\mathtt{T}(x^*-x^k)+ \theta a(x^k)^\mathtt{T}Y^{k}a(x^k)+\\
\sum_{i=1}^m2\theta y_i^k a_i(x^k)\grad a_i(x^k)^\mathtt{T}(x^*-x^k).
\end{eqnarray*}
Therefore,
\begin{eqnarray}
f(x^*)-f(x^k)\geq-\theta a(x^*)^\mathtt{T}Y^{k}a(x^*)+ \theta a(x^k)^\mathtt{T}Y^{k}a(x^k)+a(x^k)^\mathtt{T}y^k\nonumber\\
+\grad_x \Lag(x^k,y^k)^\mathtt{T}(x^*-x^k)+\sum_{i=1}^m2\theta y_i^k a_i(x^k)\grad a_i(x^k)^\mathtt{T}(x^*-x^k).\label{referee4}
\end{eqnarray}

We proceed to bound the right-hand side of \eqref{referee4}. Note that $-\theta a(x^*)^\mathtt{T}Y^{k}a(x^*)=-\sum_{i:a_i(x^*)<0}\theta y_i^ka_i(x^*)^2\geq -\sum_{i:a_i(x^*)<0}\frac{a_i(x^*)^2}{s_i^k}\theta c\mu^{k}\geq\sum_{i:a_i(x^*)<0}2a_i(x^*)c\theta\mu^{k}$, while $\theta a(x^k)^\mathtt{T}Y^{k}a(x^k)\geq0$.

As in the proof of Theorem \ref{convexbound}, we have that \eqref{akyk-bound} holds; that is,

$$a(x^k)^\mathtt{T}y^k\geq\ell\mu^{k}\|y^k\|_1-\sum_{i:a_i(x^*)<0}(2c+\ell y_i^k)\mu^{k}-c\mu^{k}.$$

Clearly, $\grad_x \Lag(x^k,y^k)^\mathtt{T}(x^*-x^k)\geq-d\mu^{k}\|y^k\|_1\|x^*-x^k\|-d\mu^{k}$ when $\|x^*-x^k\|\leq1$. To bound the last term in \eqref{referee4},  note that for $i: a_i(x^*)<0$, $-|y_i^ka_i(x^k)|\geq\frac{a_i(x^k)}{s_i^k}c\mu^{k}\geq-2c\mu^{k}$, and for $i: a_i(x^*)=0$, we have $-|a_i(x^k)y_i^k|\geq-u\mu^{k}y_i^k$ if $a_i(x^k)\geq0$ and $-|a_i(x^k)y_i^k|\geq\ell\mu^{k}y_i^k-c\mu^{k}$ if $a_i(x^k)<0$. Therefore, 
\begin{eqnarray}
\nonumber\sum_{i=1}^m2\theta y_i^k a_i(x^k)\grad a_i(x^k)^\mathtt{T}(x^*-x^k)\geq\nonumber\\
-2\theta\sum_{i=1}^m|a_i(x^k)y_i^k|\|\grad a_i(x^k)\|\|x^k-x^*\|\geq\nonumber\\
\sum_{i: a_i(x^*)<0}-4\theta c\mu^{k}\|\grad a_i(x^k)\|+\nonumber\\
\sum_{i: a_i(x^*)=0}2\theta\min\{-u\mu^{k}y_i^k,\ell\mu^{k}y_i^k-c\mu^{k}\}\|\grad a_i(x^k)\|\|x^k-x^*\|.\label{aux-bound}
\end{eqnarray}

Note that $\min\{-u\mu^{k}y_i^k,\ell\mu^{k}y_i^k-c\mu^{k}\}$ is equal to $-u\mu^{k}y_i^k$ if $y_i^k\geq\frac{c}{\ell+u}$, while it is bounded by a constant times $\mu^{k}$ otherwise. Hence, substituting all bounds obtained back into \eqref{referee4}, we get for some constant $C\geq0$ the following:

\begin{eqnarray*}
f(x^*)-f(x^k)\geq-C\mu^{k}+\ell\mu^{k}\|y^k\|_1-d\mu^{k}\|y^k\|_1\|x^*-x^k\|+\\
\sum_{i: y_i^k\geq\frac{c}{\ell+u}}-u\mu^{k}y_i^k\|\grad a_i(x^k)\|\|x^*-x^k\|.
\end{eqnarray*}
Thus, we can take $k$ large enough such that $f(x^*)-f(x^k)\geq-C\mu^{k}+\frac{\ell}{2}\mu^{k}\|y^k\|_1$.

Since $\grad\hat\Lag_*(x^*)=0$ and $\hat\Lag_*$ is star-convex, we have $\hat\Lag_*(x^k)\geq\hat\Lag_*(x^*)=f(x^*)$, giving
\begin{eqnarray}-C\mu^{k}+\frac{\ell}{2}\mu^{k}\|y^k\|_1\leq a(x^k)^\mathtt{T}y^*+\theta a(x^k)^\mathtt{T}Y^{*}a(x^k)=\nonumber\\
\sum_{i: a_i(x^*)=0}(a_i(x^k)+\theta a_i(x^k)^2)y_i^*.\label{referee5}
\end{eqnarray}
For $i: a_i(x^*)=0$, we have for $k$ large enough that $a_i(x^k)+\theta a_i(x^k)^2\leq2a_i(x^k)$ if $a_i(x^k)\geq0$ and $a_i(x^k)+\theta a_i(x^k)^2\leq\frac{1}{2}a_i(x^k)$ if $a_i(x^k)\leq0$, where $a_i(x^k)\leq u\mu^{k}-s_i^k\leq u\mu^{k}$. It follows that the right-hand side of \eqref{referee5} is bounded by a constant times $\mu^k$. Therefore, dividing by $\mu^k$ shows that $\{y^k\}$ is bounded.\jo{\qed}
\end{proof}

\begin{remark}
Given the bound $|a_i(x^k)y_i^k|\leq\max\{u\mu^{k}y_i^k,c\mu^{k}-\ell\mu^{k}y_i^k\}$ obtained in \eqref{aux-bound} for $i:a_i(x^*)=0$, assumptions 1-5 in Theorem \ref{thm:nonconvex-bound-lag} together with the assumption $\mu^{k}\|y^k\|_1\to0$ imply assumption 6 under weak constraint qualifications \cite{rcpld,cpg,ccp,strict}. Also, assumption $6$ and the star-convexity of $\hat{\Lag}_*$ in assumption $7$ imply that $x^*$ is a local solution.
\end{remark}

\begin{remark}
Although we have decided by a clearer presentation, one could get the result under a weaker assumption than Assumption~\ref{thm:nonconvex-bound-lag:final-ass} of Theorem~\ref{thm:nonconvex-bound-lag}. In particular, suppose that $\hat{\Lag}_{k}(x):=\Lag(x,y^k)+\sum_{i=1}^m \theta_i^ka_i(x)^2$. Further assume $\theta_i^k\leq C\mu^{k}$ for some $C\geq0$ when $a_i(x^*)<0$, and one of the following two conditions hold for all $i$ such that $a_i(x^*)=0$,
\begin{enumerate}
\item $\theta_i^k\le\theta(y_i^k+\|y^k\|_1 \mathcal{I}[a_i\equiv-a_j\mbox{ for some }j\neq i])$, where $\mathcal{I}[\cdot]$ is the indicator function, or
\item $\theta_i^k\le\theta\|y^k\|_1$, under a strict complementarity assumption, namely, that $\{y_i^k\}$ is bounded away from zero,
\end{enumerate}
with some $\theta>0$.
Note that by taking $\theta_i^k:=\theta y_i^k$ with condition one we subsume Assumption~\ref{thm:nonconvex-bound-lag:final-ass} of Theorem~\ref{thm:nonconvex-bound-lag}. Condition one is useful when an equality constraint $a_i(x)=0$ is represented as two inequalities $a_i(x)\leq0$ and $a_j(x):=-a_i(x)\leq0$. In that case, we may select $\theta_i^k$ and $\theta_j^k$ considerably larger, namely, proportional to the sum of all dual variables (instead of only the one correspondent to constraint $i$ and $j$, respectively). The second condition says that we may consider this larger $\theta_i^k$ for all constraints (proportional to the sum of all dual variables), as long as we have strict complementarity. 

The main modification to the proof of Theorem \ref{thm:nonconvex-bound-lag} would be on the bound of $|\theta_i^ka_i(x^k)|$ in \eqref{aux-bound}. For the first condition, with equality constraints split as two inequalities, the bound $-|a_i(x^k)y_i^k|\geq -u\mu^{k}y_i^k$ holds regardless of the sign of $a_i(x^k)$. For the second condition one gets $-|\theta_i^ka_i(x^k)|\geq\theta_i^k\min\{-u\mu^{k},\ell\mu^{k}-s_i^k\}$, and the strict complementarity assumption would give $-s_i^k\geq-\mu^{k}\frac{u}{y_i^k}$ with $\frac{u}{y_i^k}$ bounded. The result would now follow as in the proof of Theorem \ref{thm:nonconvex-bound-lag}.
\end{remark}

Note that functions $\hat\Lag_{k}$, in which we require star-convexity, are closely related to the sharp Lagrangian function \cite{rockafellar}, where we replace the $\ell_2$-norm of $a(x)$ by a weighted $\ell_2$-norm squared.


\subsection{Maximal complementarity}

We now focus our attention on obtaining maximal complementarity of the dual sequence under a set of algorithmic assumptions more general than the ones described in \eqref{seq-ipm-eq:nice-IPM}.

We say that a Lagrange multiplier $y^*$ associated with $x^*$ is maximally complementary if it has the maximum number of non-zero components among all Lagrange multipliers associated with $x^*$. Note that a maximally complementary multiplier always exists, because any convex combination of Lagrange multipliers is also a Lagrange multiplier. If a maximally complementary Lagrange multiplier $y^*$ has a component $y^*_i=0$ with $a_i(x^*)=0$, then the $i$th component of all Lagrange multipliers associated with $x^*$ are equal to zero. An interesting property of an algorithm that finds a maximally complementary Lagrange multiplier $y^*$ is that if a strictly complementary Lagrange multiplier exists, then $y^*$ satisfies strict complementarity. 

There are benefits of algorithms with iterates that converge, in a subsequence, to a point satisfying strict complementarity. In particular, strict complementarity implies the critical cone is a subspace. One can therefore efficiently check if the second-order sufficient conditions hold by checking if the matrix $\grad_x^2 \Lag (x^{*}, y^{*})$ projected onto this subspace is positive definite. This allows us to confirm strict local optimality. Furthermore, when iterates converge to a point satisfying second-order sufficient conditions, strict complementarity and Mangasarian-Fromovitz, then the assumptions of Vicente and Wright \cite{vicente2002local} hold. Therefore the IPM they studied has superlinear convergence. Our work complements theirs because they could guarantee the premise of their theorems on nonconvex problems unless the optimal dual multipliers were unique, in which case standard results prove superlinear convergence \cite{el1996formulation}.

In the next theorem, we show that if the constraint violation is reduced quickly enough relative to complementarity, then the dual sequence will be maximally complementary. To prove this result we assume either the problem is convex, or that the following ``extended'' Lagrangian function is locally star-convex
\begin{equation}
\label{extlag}
\tilde{\Lag}(x, y) := \Lag(x, y) + \theta \sum_{i:a_i(x^*)=0}(\grad a_i(x^*)^\mathtt{T}(x - x^{*}))^2,
\end{equation}
and that $\| x^k - x^* \|  \le C \sqrt{\mu^k}$ for some constant $C > 0$. 

Similar results to Theorem~\ref{thm:maxcomp} are well known when the functions are convex \cite{guler1993convergence}, and therefore our main contribution is when the functions $f$ and $a_i$ for $i=1,\dots,m$ are not convex.


\begin{theorem}\label{thm:maxcomp}
Let $\{(x^k,y^k,s^k,\mu^k)\}\subset\mathbb{R}^n\times\mathbb{R}^m\times\mathbb{R}^m\times\R$ with $\mu^{k}>0$ and $\mu^{k}\to0$ be such that:
\begin{enumerate}
\item $x^k\to x^*$  with $a(x^*)\leq0$ and $s^k\to s^*:=-a(x^*)$,
\item $y^k\geq0$ and $s^k\geq0$ with $y^k\to y^*$ ($y^*$ is necessarily a Lagrange multiplier associated with $x^*$),
\item for some $0<b\leq c$, $\mu^{k}b\leq y_i^ks_i^k$ for all $i: a_i(x^*)=0$ and $(y^k)^\mathtt{T}s^k\leq\mu^{k}c$,
\item for some $u\geq0$, $|a_i(x^k)+s_i^k|\leq\mu^{k} u$ for all $i: a_i(x^*)=0$,
\item for some $d\geq0$, $\|\grad_x \Lag(x^k,y^k)\|\leq d\mu^{k}(\|y^k\|_1+1)$,
\item \label{ssoc-convex-assumption} the functions $f$ and $a_i$ for $i=1,\dots,m$ are convex functions, or \begin{itemize}
\item[\textbullet] there is a neighborhood $S$ of $x^*$ and $W$ of $y^*$ such that for all $y\in W$, the function $\tilde{\Lag}(x,y)$ is star-convex around $x^*$ on $S$, and\\[-0.3cm] 
\item[\textbullet] there is a constant $C\geq0$ such that $\|x^k-x^*\|\leq C\sqrt{\mu^k}$.
\end{itemize}
\end{enumerate}
Then, $y^*$ is maximally complementary, i.e., $y_i^*>0$ whenever there exists some Lagrange multiplier $\tilde{y}$ associated with $x^*$ with $\tilde{y}_i>0$.
\end{theorem}

\begin{proof}
First, observe that for any Lagrange multiplier $\tilde{y}$ associated with $x^*$ we have
\begin{flalign}
\nonumber\sum_{i: a_i(x^*)=0} \frac{\tilde{y}_i}{y^k_i}&\le \sum_{i:a_i(x^*)=0}\frac{1}{\mu^{k}b} s_i^k\tilde{y}_i \\
\nonumber&=  \sum_{i:a_i(x^*)=0}\frac{1}{\mu^{k} b} \left(  s_i^k(\tilde{y}_i-y_i^k)+s_i^ky_i^k \right) \\
\nonumber&=\sum_{i:a_i(x^*)=0}\frac{1}{\mu^{k} b} \left(  (-a_i(x^k))(\tilde{y}_i-y_i^k)+(a_i(x^k)+s_i^k)(\tilde{y}_i-y_i^k)+s_i^ky_i^k \right)\\
&\le\sum_{i:a_i(x^*)=0}\frac{ a_i(x^k)(y_i^{k} - \tilde{y}_i ) }{\mu^{k} b}  + \frac{u}{b}\|y^{k} - \tilde{y}\|_1+ \frac{c}{b}. \label{eq:max-comp-bound}
\end{flalign}
If we can show that $a_i(x^{k})(y_i^{k} - \tilde{y}_i)$ is bounded by a constant times $\mu^{k}$, then the boundedness of the expression in \eqref{eq:max-comp-bound} would imply that $y_i^k$ can only converge to zero when $\tilde y_i=0$ for all Lagrange multipliers, which gives the result.
 The remainder of the proof is dedicated to showing this and separately considers the two cases given in assumption~\ref{ssoc-convex-assumption}.
 
 First, we consider the case where $f$ and $a_i$ for $i=1,\dots,m$ are convex functions. Since $\grad_x\Lag(x^*,\tilde{y})=0$, we have $\Lag(x^k,\tilde{y})\geq\Lag(x^*,\tilde{y})$, and thus,
\begin{flalign}
\nonumber( a(x^k)  - a(x^{*}) )^\mathtt{T} (y^k - \tilde{y}) &= \left( {\Lag}(x^{*}, \tilde{y}) - {\Lag}(x^{*}, y^k)  \right) + \left( {\Lag}(x^k, y^k) -  {\Lag}(x^{k}, \tilde{y}) \right) \\
\nonumber&\le {\Lag}(x^{k}, y^{k}) -  {\Lag}(x^{*}, y^k)  \\
&\le \grad_x {\Lag}(x^{k}, y^{k})^\mathtt{T} (x^{k} - x^{*} ),\label{eq:lag-inequality-final}
\end{flalign}
where the last inequality uses the convexity of $\Lag(x, y^{k})$ with respect to $x$. Since 
$$(a(x^k)  - a(x^{*}) )^\mathtt{T} (y^k - \tilde{y})=\sum_{i:a_i(x^*)=0}a_i(x^k)(y_i^k-\tilde{y}_i)+\sum_{i:a_i(x^*)<0}(a_i(x^k)-a_i(x^*))y_i^k$$ and $a_i(x^*)y_i^k\leq0$, we have $$\sum_{i:a_i(x^*)=0}a_i(x^k)(y_i^k-\tilde{y}_i)\leq\grad_x {\Lag}(x^{k}, y^{k})^\mathtt{T}(x^k-x^*)-\sum_{i:a_i(x^*)<0}a_i(x^k)y_i^k.$$

It remains to bound the right-hand side of the previous expression. For $i:a_i(x^*)<0$ we have $-a_i(x^k)y_i^k\leq \frac{-a_i(x^k)}{s_i^k}\mu^{k}c\leq2\mu^{k}c$.
Also, $\|\grad_x {\Lag}(x^{k}, y^{k})\|\leq d\mu^{k}(\|y^k\|_1+1)\leq d\mu^{k}(\|y^*\|_{1} +2)$. This concludes the proof in the convex case.

On the other hand, let us assume the remaining conditions in assumption 6. We note first that we can take the Lagrange multiplier $\tilde{y}$ sufficiently close to $y^*$ without loss of generality because for any Lagrange multiplier $\hat{y}$ associated with $x^*$ we can take $\tilde{y}$ of the form $\tilde{y} := \eta \hat{y} + (1 - \eta) y^{*}, \eta\in(0,1)$, with the property that if $\hat{y}_{i} > 0$ then $\tilde{y}_{i} > 0$. Now, similarly to \eqref{eq:lag-inequality-final}, from the star-convexity of $\tilde{\Lag}(x,\tilde{y})$ and $\tilde{\Lag}(x,y^k)$ we have
\begin{flalign}
\nonumber( a(x^k)  - a(x^{*}) )^\mathtt{T} (y^k - \tilde{y}) &= \left( \tilde{\Lag}(x^{*}, \tilde{y}) - \tilde{\Lag}(x^{*}, y^k)  \right) + \left( \tilde{\Lag}(x^k, y^k) -  \tilde{\Lag}(x^{k}, \tilde{y}) \right) \\
\nonumber&\le \tilde{\Lag}(x^{k}, y^{k}) -  \tilde{\Lag}(x^{*}, y^k)  \\
&\le \grad_x \tilde{\Lag}(x^{k}, y^{k})^\mathtt{T} (x^{k} - x^{*} ). \label{eq:lag-inequality-nonconvex}
\end{flalign}
Hence, $$\sum_{i:a_i(x^*)=0}a_i(x^k)(y_i^k-\tilde{y}_i)\leq\grad_x \tilde{\Lag}(x^{k}, y^{k})^\mathtt{T}(x^k-x^*)-\sum_{i:a_i(x^*)<0}a_i(x^k)y_i^k.$$
It remains to bound the right-hand side of the previous expression by a constant times $\mu^k$.
Note that $-a_i(x^k)y_i^k\leq2\mu^{k}c$ for $i:a_i(x^*)<0$ and $$\|\grad_x\tilde{\Lag}(x^k,y^k)\|\leq d\mu^{k}(\|y^k\|_1+1)+2\theta\sum_{i: a_i(x^*)=0}\|\grad a_i(x^*)\|^2 \|x^k-x^*\|.$$ 
The result now follows from the bound $\|x^k-x^*\|\leq C\sqrt{\mu^k}$.\jo{\qed}
\end{proof}

The next lemma shows that one can guarantee the upper bound on $\{\|x^k-x^*\|\}$ given in assumption 6 of Theorem \ref{thm:maxcomp} by assuming the standard second-order sufficient condition.

\begin{lemma}[\citet{hager2014}]\label{lem:distance}Let $f$ and $a$ be twice differentiable at a local minimizer $x^*$ with a Lagrange multiplier $y^*\in\R^m$ satisfying the sufficient second-order optimality condition: \begin{eqnarray}\label{ssoc}d^\mathtt{T}\grad^2_{xx}\Lag(x^*,y^*)d\geq\lambda\|d\|^2, \mbox{ for all }d\mbox{ such that }\\
\grad f(x^*)^\mathtt{T}d\leq0, \grad a_i(x^*)^\mathtt{T}d\leq0, i: a_i(x^*)=0,\nonumber\end{eqnarray}
for some $\lambda>0$. Then, there is a neighborhood $\mathcal{B}$ of $(x^*,y^*,-a(x^*))$ such that if $(x,v,s)\in\mathcal{B}$ with $v\geq0$, $v_i=0$ for $i: a_i(x^*)<0$ and $s\geq0$, we have $$\|x-x^*\|\leq C\sqrt{\max\{\|\grad_x\Lag(x,v)\|,\|[a(x)+s]_{i:a_i(x^*)=0}\|,v^\mathtt{T}s\}}$$ for some $C\geq0$.
\end{lemma}
\begin{proof}
The result follows from \cite[Theorem 4.2]{hager2014} because the sufficient optimality condition is equivalently stated at constraints $a(x)\leq0$ or at the slack variable formulation $a(x)+s=0, s\geq0$. Inactive constraints are removed from the problem and equivalence of norms is employed.\jo{\qed}
\end{proof}

Another useful result is the following.


\begin{lemma}[\citet{debreu1952definite}]\label{lem:well-known-result}
Let $H \in \R^{n \times n}$ be a symmetric matrix and $A \in \R^{m \times n}$.
If $d^T H d > 0$ for all $d \in \R^{n}$ such that $A d = 0$, then there exists $\theta \ge 0$ such that $H + \theta A^T A \succ 0$.
\end{lemma}

Now we can replace our nonconvex assumptions in Theorem \ref{thm:maxcomp} by the second-order sufficiency condition as follows.

\begin{theorem}\label{thm:maxcomp-suf}
Let $\{(x^k,y^k,s^k,\mu^k)\}\subset\mathbb{R}^n\times\mathbb{R}^m\times\mathbb{R}^m\times\R$ with $\mu^{k}>0$ and $\mu^{k}\to0$ be such that:
\begin{enumerate}
\item $x^k\to x^*$  with $a(x^*)\leq0$ and $s^k\to s^*:=-a(x^*)$,
\item $y^k\geq0$ and $s^k\geq0$ with $y^k\to y^*$ ($y^*$ is necessarily a Lagrange multiplier associated with $x^*$),
\item for some $0<b\leq c$, $\mu^{k}b\leq y_i^ks_i^k$ for all $i: a_i(x^*)=0$ and $(y^k)^\mathtt{T}s^k\leq\mu^{k}c$,
\item for some $u\geq0$, $|a_i(x^k)+s_i^k|\leq\mu^{k} u$ for all $i: a_i(x^*)=0$,
\item for some $d\geq0$, $\|\grad_x \Lag(x^k,y^k)\|\leq d\mu^{k}(\|y^k\|_1+1)$,
\item $f$ and $a$ are twice continuously differentiable and $(x^*,y^*)$ satisfies the sufficient second-order optimality condition \eqref{ssoc}.
\end{enumerate}
Then, $y^*$ is maximally complementary, i.e., $y_i^*>0$ whenever there exists some Lagrange multiplier $\tilde{y}$ associated with $x^*$ with $\tilde{y}_i>0$.
\end{theorem}

\begin{proof}
Since the sufficient second-order optimality condition holds at $(x^{*}, y^{*})$ by  Lemma~\ref{lem:well-known-result}, there exists $\theta \geq 0$ such that
$$
\grad^2_{x,x} \tilde{\Lag}(x^{*}, y^{*}) = \grad^2_{x,x} \Lag(x^{*}, y^{*}) + \theta \sum_{i:a_i(x^*)=0} \grad a_i(x^*)^\mathtt{T} \grad a_i(x^*) \succ 0.
$$
It follows that there exists some neighborhood $\mathcal{B}$ of $(x^*,y^*)$ such that $\tilde{\Lag}(x,y)$ is convex on $x$ for all $(x,y)\in\mathcal{B}$.


Also, for $v_i^k:=y_i^k$ if $a_i(x^*)=0$ and $v_i^k:=0$ otherwise, we have $$\|\grad_x\Lag(x^k,v^k)\|\leq\|\grad_x\Lag(x^k,y^k)\|+\|\sum_{i:a_i(x^*)<0}y_i^k\grad a_i(x^k)\|,$$ which is bounded by a non-negative constant times $\mu^{k}$. By Lemma~\ref{lem:distance} we have $\|x^k-x^*\|\leq C\sqrt{\mu^{k}}$ for some constant $C\geq0$. Hence, the result follows by Theorem \ref{thm:maxcomp}.\jo{\qed}
\end{proof}

From the proof of Theorem~\ref{thm:maxcomp-suf} we can see the term $\theta \sum_{i:a_i(x^*)=0}(\grad a_i(x^*)^\mathtt{T}(x - x^{*}))^2$ in \eqref{extlag} is important because it guarantees $\tilde{\Lag}(x,y^{*})$ is convex if the second-order sufficient conditions hold. Conversely, even if the second-order sufficient conditions hold, the Lagrangian $\Lag(x,y^{*})$ may not be convex in a neighborhood of this point. For example, consider the problem $\min{-x^2}$ s.t. $x \ge 0, x \le 0$ at the point $x = 0$; the second-order sufficient conditions are satisfied, but the Lagrangian is not convex in $x$. However, as we show in Theorem~\ref{thm:maxcomp-suf}, the second-order sufficient conditions imply the nonconvex case of assumption~\ref{ssoc-convex-assumption} of Theorem~\ref{thm:maxcomp}.

Now that Theorem~\ref{thm:maxcomp} and \ref{thm:maxcomp-suf} are proved, we discuss possible extensions. 
When there are additional constraints $\tilde{a}_i(x)\leq0, i=1,\dots,\tilde{m}$ that are known to have a strict interior (for instance, if they represent simple bounds on the variables), a common implementation choice is to maintain feasibility for these constraints at each iteration, instead of considering the slow reduction of feasibility suggested by \eqref{seq-ipm-eq:primal-rate}. Note that assumption $4$ of Theorem \ref{thm:maxcomp} is weaker than \eqref{seq-ipm-eq:primal-rate} and includes the possibility of keeping $\tilde{a}_i(x^k)+s_i^k=0, i=1,\dots,\tilde{m}$, at each iteration. With respect to the results of Theorems \ref{convexbound} and \ref{thm:nonconvex-bound-lag}, one may weaken their assumption $4$ in order to consider the case $\tilde{a}_i(x^k)+s_i^k=0, i=1,\dots,\tilde{m}$, by strengthening the corresponding assumption $5$ by replacing the term $\|y^k\|_1$ on the bound of $\|\grad_x\Lag(x^k,y^k)\|$  (which includes all dual multipliers) by the possibly smaller sum of the multipliers associated only with the original constraints $a_i(x)\leq0$.

\section{When things may fail}\label{negative}

We now limit our results to the convex case, where we explore the possibility of \eqref{seq-ipm-eq:primal-rate} not being satisfied (i.e., the constraint violation is not reduced at the same rate as complementarity).

In the following theorem, we show that controlling the constraint violation rate is essential for the boundedness of the dual sequence. In fact, we show that if the constraint violation reduces faster than the barrier parameter $\mu^{k}$, the dual sequence is unbounded whenever the constraints are convex and the set of Lagrange multipliers is unbounded. We note that a similar result was already known when the functions $f$ and $a$ are linear \cite[Theorem 4]{mizuno1995surface}.

\begin{theorem}
\label{thm:divergence1}
Assume that $a$ is convex and the feasible region has empty interior. Let $\{(x^k,y^k,s^k,\mu^k)\}\subset\mathbb{R}^n\times\mathbb{R}^m\times\mathbb{R}^m\times\R$ with $\mu^{k}>0$ for all $k$ and $\mu^{k}\to0$ be such that:
\begin{enumerate}
\item $x^k\to x^*$  with $a(x^*)\leq0$ and $s^k\to-a(x^*)$,
\item $y^k\geq0$ and $s^k\geq0$,
\item for some $b>0$, $\mu^{k}b\leq y_i^ks_i^k$ for all $i: a_i(x^*)=0$,
\item $\frac{a_i(x^k)+s_i^k}{\mu^{k}}\to0$ for all $i: a_i(x^*)=0$.
\end{enumerate}
Then $\{y^k\}$ is unbounded.
\end{theorem}

\begin{proof}
Note that there is no $d\in\mathbb{R}^n$, $d \neq 0$ with $\grad a_i(x^*)^\mathtt{T} d<0$ for all $i: a_i(x^*)=0$, otherwise, $x^*+td$ would be interior for $t>0$ sufficiently small. By Farkas's Lemma, there is some $\hat{y}\in\mathbb{R}^m$ with $\hat{y}\geq0, \hat{y}\neq0$, $a(x^*)^\mathtt{T}\hat{y}=0$ and $\sum_{i=1}^m\hat{y}_i\grad a_i(x^*)=0$.
For all $i$, we have $a_i(x^k)\geq a_i(x^*)+\grad a_i(x^*)^\mathtt{T}(x^k-x^*)$ and hence $a(x^k)^\mathtt{T}\hat{y}\geq a(x^*)^\mathtt{T}\hat{y}+\sum_{i=1}^m\hat{y}_i\grad a_i(x^*)^\mathtt{T}(x^k-x^*)=0$. Thus, $$\hat{y}^\mathtt{T}(a(x^k)+s^k)=\hat{y}^\mathtt{T}a(x^k)+\hat{y}^\mathtt{T}s^k\geq\hat{y}^\mathtt{T}s^k.$$
Take $i$ such that $\hat{y}_i>0$ and we have $$0<\mu^{k}b\hat{y}_i\leq y_i^ks_i^k\hat{y}_i\leq y_i^k\hat{y}^\mathtt{T}s^k\leq y_i^k\hat{y}^\mathtt{T}(a(x^k)+s^k).$$ Then, $\hat{y}^\mathtt{T}(a(x^k)+s^k)>0$ and $y_i^k\geq b\hat{y}_i\frac{\mu^{k}}{\hat{y}^\mathtt{T}(a(x^k)+s^k)}\to+\infty$.\jo{\qed}
\end{proof}

The next theorem shows that the dual sequence can have a poor quality in terms of maximal complementarity if constraint violation is not reduced fast enough. We prove that in this instance the dual sequence limits to a point with minimal complementarity.

\begin{theorem}\label{mincomp}
Let $f$ and $a_i$ for $i=1,\dots,m$ be convex functions and $\{(x^k,y^k,s^k,\mu^k)\}\subset\mathbb{R}^n\times\mathbb{R}^m\times\mathbb{R}^m\times\R$ with $\mu^{k}>0$ and $\mu^{k}\to0$ be such that:
\begin{enumerate}
\item $x^k\to x^*$  with $a(x^*)\leq0$ and $s^k\to s^*:=-a(x^*)$,
\item $y^k\geq0$ and $s^k\geq0$ with $y^k\to y^*$ ($y^*$ is necessarily a Lagrange multiplier associated with $x^*$),
\item for some $c\geq0$, $(y^k)^\mathtt{T}s^k\leq\mu^{k}c$,
\item $0\leq a_i(x^k)+s_i^k$ for all $i: a_i(x^*)=0$,
\item for some $d\geq0$, $\|\grad_x \Lag(x^k,y^k)\|\leq d\mu^{k}(\|y^k\|_1+1)$.

\end{enumerate}
Let $\tilde{y}\in\R^m$ be some Lagrange multiplier associated with $x^*$ such that  for all $i: a_i(x^*)=0$, \begin{itemize}
\item[\textbullet] $\frac{a_i(x^k)+s_i^k}{\mu^{k}}\to+\infty$ when $\tilde{y}_i=0$, and 
\item[\textbullet] $a_i(x^k)+s_i^k\leq u\mu^{k}\mbox{ or }y_i^k\geq\tilde{y}_i$ when $\tilde{y}_i>0$,
\end{itemize}
for some $u\geq0$.
Then, $y_i^*=0$ whenever $\tilde{y}_i=0$. In particular, if $\tilde{y}$ is minimally complementary, that is, it has a minimal number of non-zero elements, then $y^*$ is also minimally complementary.
\end{theorem}

\begin{proof}
Let $\tilde{y}$ be a Lagrange multiplier associated with $x^*$. We have



$$\sum_{i:a_i(x^*)=0}\frac{1}{\mu^{k}} (a_i(x^k)+s_i^k)(y_i^k-\tilde{y}_i)=\frac{1}{\mu^{k}}\sum_{i:a_i(x^*)=0}s_i^k{y}_i^k+a_i(x^k)(y_i^k-\tilde{y}_i)-s_i^k\tilde{y}_i.$$
Since $s_i^k\tilde{y}_i\geq0$, $(s^k)^\mathtt{T}y^k\leq\mu^{k}c$ and, from the proof of Theorem \ref{thm:maxcomp}, $a_i(x^k)(y_i^k-\tilde{y}_i)\leq C\mu^{k}$ for some $C\geq0$, we have
$$\sum_{i:a_i(x^*)=0}\frac{a_i(x^k)+s_i^k}{\mu^{k}}(y_i^k-\tilde{y}_i)\leq c+C,$$
and the result follows.\jo{\qed}
\end{proof}

If assumption $4$ in Theorem \ref{mincomp} is replaced by a similar one with a strict inequality, and assumption $5$ is replaced by $\grad\Lag(x^k,y^k)^\mathtt{T}(x^k-x^*)\leq d\mu^{k}$ for some $d\geq0$, then we can drop the assumption that $\{y^k\}$ is convergent. It will then follow that $\{y^k\}$ is bounded, and any limit point $y^*$ will have the property stated in the theorem.

In the next section we investigate the numerical behavior of the dual sequences generated by IPOPT on the NETLIB collection.

\section{Numerical experiments}\label{sec:experiments}

In this section, we contrast a well-behaved IPM, the one-phase IPM \cite{hinder2018one} that satisfies \eqref{seq-ipm-eq:nice-IPM}, with IPOPT, an IPM that \emph{tends} to moves the primal feasibility faster than \eqref{seq-ipm-eq:nice-IPM} would suggest. Empirically, we demonstrate on both linear and nonlinear programs that IPOPT has issues with the dual multiplier norms exploding, but the one-phase IPM does not. This demonstrates that our theory has practical implications for the design of IPMs. 


Many IPM codes, such as IPOPT, keep $\frac{s_i y_i}{\mu}$ bounded below and require an inequality similar to
$$
\| \grad_{x} \Lag(x,y) \| +  \| a(x) + s \| + \max_{i} s_i y_i \le \mu ( 1 + \| y \| )
$$
to hold before $\mu$ is decreased \cite[Algorithm~19.1]{nocedal2006numerical}. Hence assumptions 3--5 of Theorem~\ref{thm:maxcomp-suf} hold, and it follows that the IPM iterates are likely to converge to a maximal complementarity solution. 

Our tests do not include IPMs that risk not tending to a minimal complementarity solution, i.e., reduce the constraint violation slower than perturbed complementarity. However, such IPMs certainly could be artificially created. This phenomenon might also occur naturally, for example, in dual regularized IPMs \cite{altman1999regularized} or $\ell_2$-penalty IPMs \cite{chen2006interior} \emph{if} the algorithm is not well-designed.


The code for replicating our results can be found at \url{https://github.com/ohinder/Lagrange-multipliers-behavior.jl}. We test on the NETLIB test set of real linear programs in Section~\ref{sec:linear-programs} and then on three toy nonconvex programs in Section~\ref{sec:nonconvex-programs}.

\subsection{Linear programs}\label{sec:linear-programs}

The focus of this section is showing that on the NETLIB test set -- of real linear programming problems -- IPMs such as IPOPT, that aggressively reduce the primal feasibility, will have unnecessarily large dual iterates. As we discussed in the introduction, the convergence analysis of IPOPT and many other nonlinear optimization solvers \cite{byrd2000trust,wachter2005line} assumes that the set of dual multipliers at the convergence point is bounded to guarantee that the dual multipliers do not diverge. One natural question is whether these assumptions are valid on a test set like NETLIB. As documented in Table~\ref{tbl:strict-interior} in the appendix, we find that $64$ of the $95$ linear programs we tested lack a strict relative interior, and therefore Mangasarian-Fromovitz constraint qualification fails to hold. See the Appendix for more details on the experiments.

The next natural question is to check if the violation of these assumptions translates into undesirable behavior on these test problems. Consider Figure~\ref{fig:netlib-example} where we plot the performance of IPOPT on the problem ADLITTLE from the NETLIB collection. As our theory predicts when the primal feasibility is reduced faster than complementarity, the dual variables increase substantially. When IPOPT's default perturbation strategy is used, while the final dual variable value is only $3 \times 10^3$, the maximum dual variable value still spikes to $4 \times 10^7$ on iteration $22$. This contrasts with the one-phase IPM \cite{hinder2018one} that smoothly reduces the constraint violation, dual feasibility, and complementarity; consequently, the maximum dual variable follows a smooth trajectory.

\begin{figure}
\centering
\includegraphics[scale=0.55]{\figuresDir{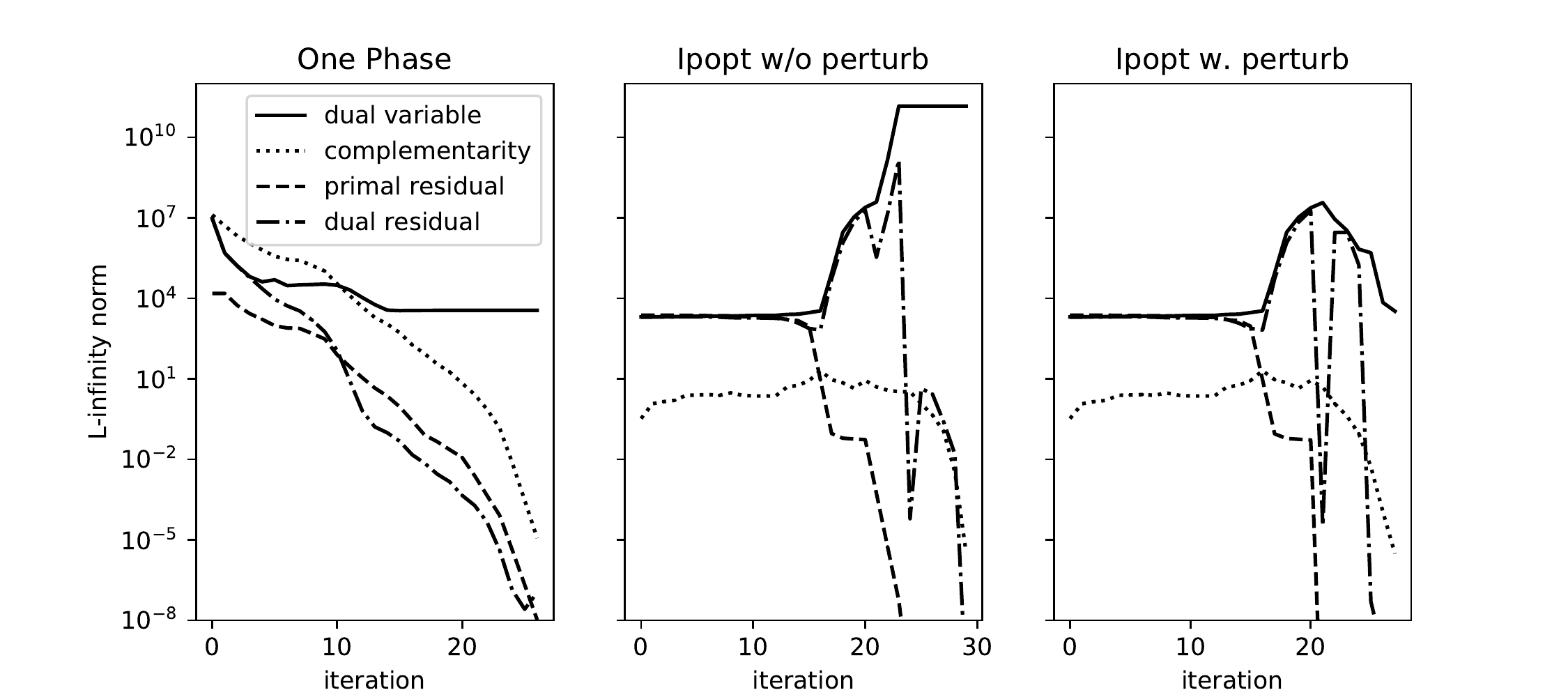}}
\caption{Comparison of the iterates of different IPMs on the NETLIB problem ADLITTLE.}\label{fig:netlib-example}
\end{figure}


Next, we show that this phenomenon occurs across the whole NETLIB test set. We run these IPMs on the NETLIB problems with less than $10,000$ non-zero entries and record the maximum dual variable value (across all the IPMs iterates). All solvers successfully terminate, within the maximum number of iterations of \maxIter, on $56$  of the $68$ problems. See the Appendix for further details. Figure~\ref{fig:dual-maxes-netlib} plots an empirical cumulative distribution over the maximum dual variable for each solver. In particular, for each solver, it plots the function $g : [0,1] \rightarrow \R$ where $g(\theta)$ is the maximum dual variable value of the problem, for which, exactly a $\theta$ proportion of the problems have a smaller or equal maximum dual variable value. The plot illustrates that the maximum dual variable of IPOPT in the last few iterations (either with or without the default perturbation) is unnecessarily large for most problems that lack a strict relative interior.

\begin{figure}[H]
\centering
\includegraphics[scale=0.6]{\figuresDir{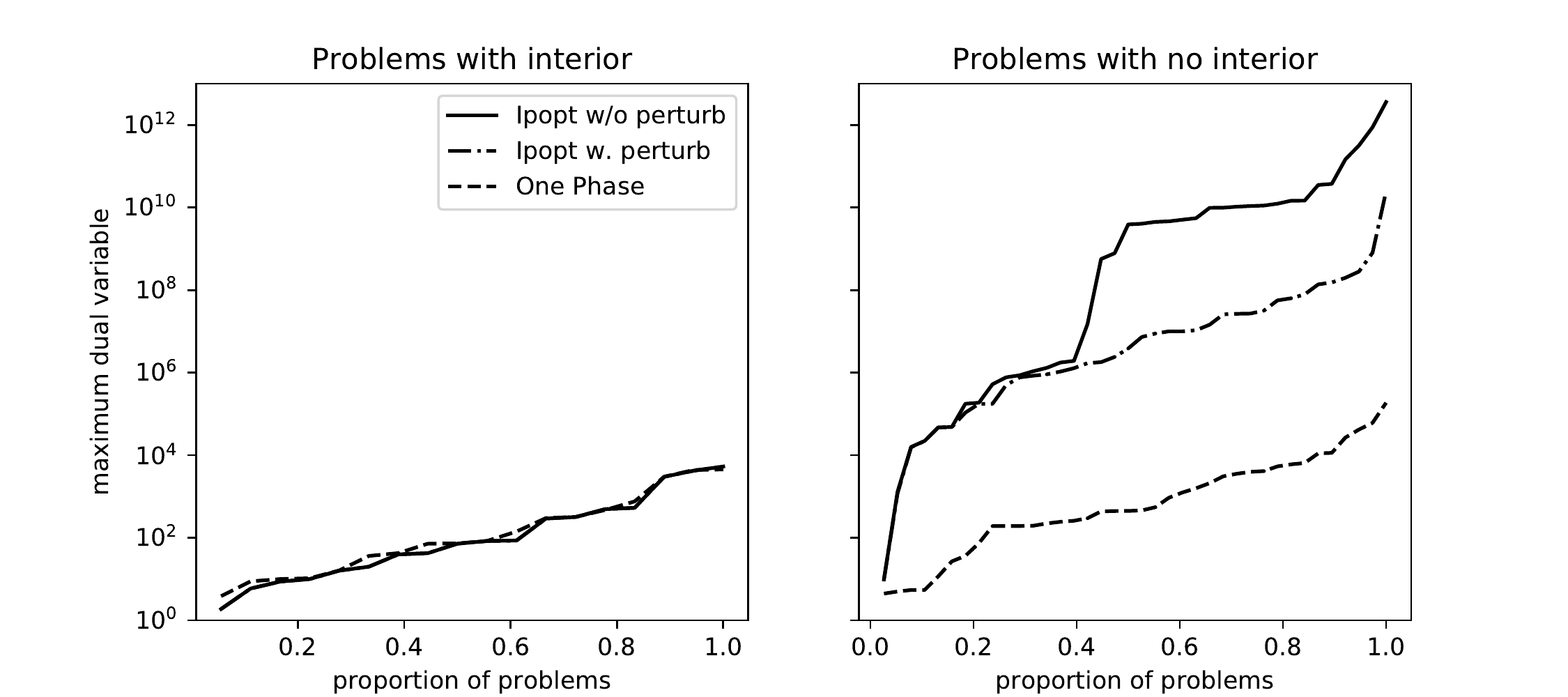}}
\caption{Comparison of the maximum dual variable value over the last 20\% of iterations for different IPMs on the NETLIB collection.}\label{fig:dual-maxes-netlib}
\end{figure}

\subsection{Nonconvex programs}\label{sec:nonconvex-programs}

This section focuses on nonconvex programs. We test IPOPT and the one-phase IPM on three toy examples. The results for these examples are given in Table~\ref{nonlinear-table}, and we believe validate the utility of our theory. Additional figures displaying the algorithm trajectories are given in Appendix~\ref{sec:nonconvex-details}. The first two problems were chosen to satisfy the assumptions of our theory. The final problem gives an example, derived from issues encountered in drinking water network optimization, where dual multipliers exploding is a practical issue.

\paragraph{Intersection of two circles.} This problem is written as
\begin{subequations}\label{circle-equality-example-formulation}
\begin{flalign}
\minimize & -(x_1-1)^2 + x_2^2 \\
\subjectTo & ~x_1^2 + x_2^2 \le 1 \\
 & ~(x_1-2)^2 + x_2^2 \le 1.
\end{flalign}
\end{subequations}
The constraints require the solution to lie in the intersection of two circles, and the objective is a nonconvex quadratic. At the optimal solution (and only feasible solution) given by $x_1 = 1$, $x_2=0$, the Mangasarian-Fromovitz constraint qualification (MFCQ) does not hold. However, the point is a KKT point. Furthermore, the set of dual multipliers corresponding to this KKT point, contains both the point $(1,1)$ which satisfies strict complementarity and the point $(0,0)$ which does not satisfy strict complementarity. 

As we show next, this problem satisfies the assumptions of Theorem~\ref{thm:nonconvex-bound-lag} and Theorem~\ref{thm:maxcomp-suf} at the point $x_1 = 1$, $x_2=0$. A picture of this problem is given in Figure~\ref{fig:circle}.
\begin{figure}
\includegraphics[scale=0.23]{\figuresDir{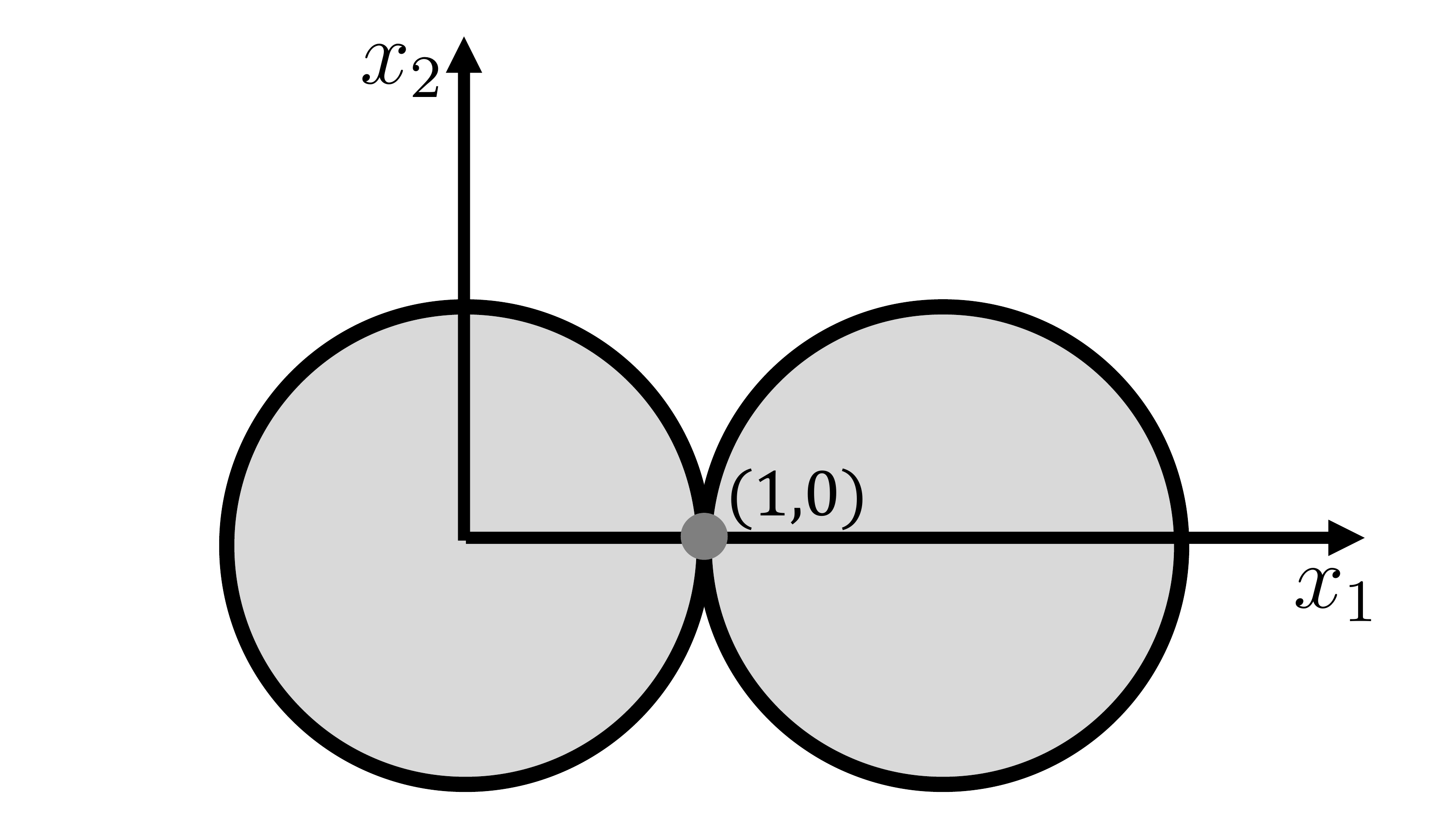}}
\caption{Picture of the circle intersection problem given in \eqref{circle-equality-example-formulation}}\label{fig:circle}
\end{figure}
Next, we verify that the assumptions of Theorem~\ref{thm:nonconvex-bound-lag} are met. Recall $x^{*} = (1,0)$. Observe, the Lagrangian, its gradient and its Hessian are
\begin{flalign*}
\Lag(x,y) &= -(x_1-1)^2 +  y_1 (x_1^2 + x_2^2 - 1) + y_2 ((x_1-2)^2 + x_2^2 - 1)  \\
\frac{\partial \Lag(x,y)}{\partial x_{1}} &=-2 (x_1-1) +  2 y_1 x_1 + 2  y_2 (x_1 - 2)  \\
\frac{\partial \Lag(x,y)}{\partial x_{2}} &=    2 (y_1 + y_2) x_2 \\
\grad_{xx}^2 \Lag(x,y) &= 
2 \begin{pmatrix} y_1 + y_2 - 1 & 0 \\
0 & y_1 + y_2 + 1
 \end{pmatrix}.
\end{flalign*}
From this we observe assumption 6 holds with $y^{*} = (1,1)$ and $\grad_{x} \Lag(x^{*},y^{*}) = 0$. Furthermore, from $\grad_{xx}^2 \Lag(x,y)$ we deduce $\Lag(x,y)$ is convex in $x$ if $y_1 + y_2 \ge 1$. This verifies assumption~7 of Theorem~\ref{thm:nonconvex-bound-lag} with $\theta = 0$. The remaining assumptions of Theorem~\ref{thm:nonconvex-bound-lag} are naturally satisfied by the one-phase IPM.

Next, we verify the assumptions of Theorem~\ref{thm:maxcomp-suf}. Since $y^k$ is bounded, there exists a convergent subsequence with limit $y^{*}$, where $y^{*}$ satisfies $\grad_{x} \Lag(x^{*}, y^{*}) = 0$. Furthermore, $d^T \grad_{xx}^2 \Lag(x^{*},y^{*}) d \ge \lambda \| d \|_2^2$ on the null space of the Jacobian of the constraints ($d_1 = 0$).

\paragraph{Linear program with complementarity constraints.} This problem is written as
\begin{subequations}\label{comp-example-formulation}
\begin{flalign}
\minimize & ~ 3 x_1 - 2 x_2 \\
\subjectTo & ~ x_1 + 3 x_2 \le 2 \\
& ~ x_1 x_2 \le 0 \\
& ~ x_1, x_2 \ge 0.
\end{flalign}
\end{subequations}
At the unique local optima given by $x_1 = 2$ and $x_2 = 0$, MFCQ does not hold. However, the point is a KKT point.
It is straightforward to see that this problem satisfies the assumptions of Corollary~\ref{coro:mpec-bound-lag} and Theorem~\ref{thm:maxcomp-suf}. The fact that Corollary~\ref{coro:mpec-bound-lag} holds is immediate because linear functions are convex. Theorem~\ref{thm:maxcomp-suf} requires verifying the second-order sufficient conditions hold. They do because the null space of the Jacobian of the constraints evaluated at the solution $x_1 = 2$, $x_2 = 0$ only contains zero.

Therefore, our theory proves that for the one-phase IPM the dual multipliers remain bounded and strict complementarity holds for both the `intersection of two circles' and `linear program with complementarity constraints' problems. Table~\ref{nonlinear-table} demonstrates this behavior is seen in practice. Table~\ref{nonlinear-table} also shows IPOPT has issues with the dual multiplier values exploding on these problems. Both solvers maintain strict complementarity for these problems.

\paragraph{Drinking water network optimization.} The final example is a toy drinking water network optimization problem (see \cite{burgschweiger2009optimization} for the formulation of real drinking water network optimization problems as nonlinear programs). The aim is to choose the minimum inlet pressure to ensure that minimum node pressures and demand for water are met. We stumbled across this example when experimenting with our one-phase IPM \cite{hinder2018one} on real drinking water networks. A diagram representing the water network is given in Figure~\ref{fig:water}.
\begin{subequations}\label{drink-example-formulation-1}
\begin{flalign}
\minimize & ~h_1 \\
\subjectTo & ~x_{1,2} + x_{1,3} = 2 \label{eq:supply-1} \\
& ~x_{1,2} + x_{2,3} = 1 \label{eq:demand-2} \\
& ~x_{1,3} = 1 \label{eq:demand-3} \\
& ~x_{1,2}^{1.8} = h_1 - h_2 \label{eq:loss-1} \\
& ~x_{1,3}^{1.8} = h_1 - h_3 \label{eq:loss-2} \\
& ~x_{2,3}^{1.8} = h_2 - h_3 \label{eq:loss-3} \\
& ~h_1, h_2, h_3, x_{1,2}, x_{1,3}, x_{2,3} \ge 0
\end{flalign}
Equation~\eqref{eq:supply-1} states that 2 units of water are available at node $1$. Equation~\eqref{eq:demand-2} and \eqref{eq:demand-3} states that 1 unit of water is demanded at both node $2$ and node $3$. Finally, \eqref{eq:loss-1}, \eqref{eq:loss-2}, \eqref{eq:loss-3} represent the pressure loss in pipes due to friction. Our objective minimizing the inlet pressure is equivalent to minimizing the shafting speed of a variable speed pump at node $1$.

\end{subequations}
The optimal solution (and unique local minimizer) occurs at $x_{1,2} = 1, x_{1,3} = 1, x_{2,3} = 0, h_1 = 0.29, h_2 = 0, h_3 = 0$. At this point, MFCQ fails but nonetheless the solution is a KKT point. We included this problem because it corresponds to a `physically meaningful' problem where MFCQ fails to hold at the optimal solution\footnote{A popular misconception is that when the constraints of a optimization problem are defined by `physics', MFCQ always holds. This is a nice counter-example.}. Table~\ref{nonlinear-table} shows that the one-phase IPM keeps the dual variables bounded but fails to maintain strict complementarity for this problem. On the other hand, IPOPT seems to have issues with both the dual multipliers exploding and strict complementarity failing.
\begin{figure}[H]
\includegraphics[scale=0.25]{\figuresDir{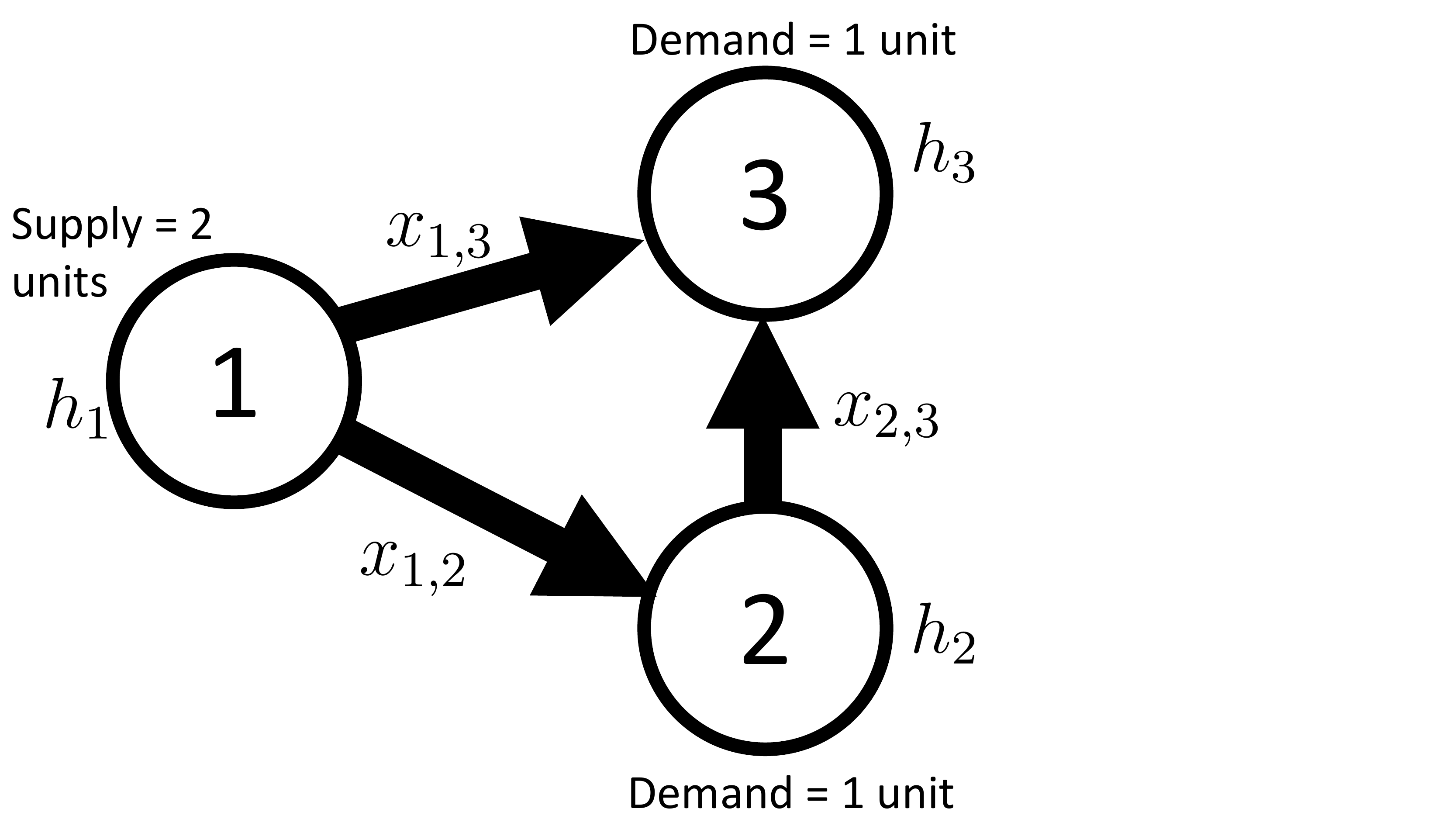}}
\caption{Picture of the drinking water network optimization problem given in \eqref{drink-example-formulation-1}. Flows across the edges are given by the $x$ variables and pressures at nodes by the $h$ variables.}
\label{fig:water}
\end{figure}

\begin{table}[H]
\centering
\begin{tabular}{c c c  c }
\bottomrule
\multicolumn{4}{c}{Intersection of two circles} \\ 
solvers & iterations & max dual & strict complementarity\\ 
 \hline
One Phase & 6 & $ 1.2 $ & $ 1.2 $\\ 
 Ipopt w/o perturb & 26 & $ 6.6 \times 10^{9} $ & $ 4.0 \times 10^{8} $\\ 
 Ipopt w. perturb & 19 & $ 7.3 \times 10^{6} $ & $ 9.1 $\\ 
\bottomrule
\multicolumn{4}{c}{Linear program with complementarity constraints} \\ 
solvers & iterations & max dual & strict complementarity\\ 
 \hline
One Phase & 8 & $ 1.2 \times 10^{1} $ & $ 8.5 \times 10^{-1} $\\ 
 Ipopt w/o perturb & 28 & $ 5.8 \times 10^{9} $ & $ 2.0 $\\ 
 Ipopt w. perturb & 25 & $ 3.6 \times 10^{5} $ & $ 2.0 $\\ 
\bottomrule
\multicolumn{4}{c}{Drinking water network optimization} \\ 
solvers & iterations & max dual & strict complementarity\\ 
 \hline
One Phase & 9 & $ 1.7 $ & $ 2.9 \times 10^{-1} $\\ 
 Ipopt w/o perturb & 6* & $ 1.2 \times 10^{4} $ & $ 1.1 \times 10^{-7} $\\ 
 Ipopt w. perturb & 44* & $ 1.9 \times 10^{3} $ & $ 1.2 \times 10^{-6} $\\ 
 \bottomrule
\end{tabular} 
\caption{A selection of nonlinear programming problems for testing dual multiplier behavior. Suppose the algorithm is generating a sequence of primal iterates $x^k$, slack iterates $s^k$, and dual iterates $y^k$. `Max dual' refers to the value $\| y^k \|_{\infty}$ over the last 20\% of iterations. `Strict complementarity' refers to the minimum value of $\min_i y^k_i + s^k_i$ over the last 20\% of iterations. 
A * indicates on these problems the `dual multiplier calculator in IPOPT failed and therefore the algorithm terminated unsuccessfully. See Section~\ref{sec:nonconvex-details} for plots of IPM trajectories for these problems.
}\label{nonlinear-table}
\end{table}

\section{Final remarks}\label{conclusions}

We demonstrated that carefully controlling both primal feasibility and the barrier parameter are important when designing IPMs to ensure the dual multipliers are well-behaved. In the linear programming community, there was awareness of this issue \cite{mizuno1995surface}, and thus, many implemented IPMs move primal feasibility and complementarity at the same rate \cite{andersen2000mosek,mehrotra1992implementation}. However, in the general nonlinear programming community, there is a lack of awareness of this issue. Consequently, there are few papers (e.g., \citet*{hinder2018one}) that consider the relative rate of reduction of primal feasibility and complementarity. 

\section{Acknowledgements}

We would like to thank the anonymous referees for their helpful feedback, and Michael Saunders for carefully proof reading the paper.

\si{
\bibliographystyle{apalike}
}
\jo{
\bibliographystyle{plainnat}
}
\bibliography{library-one-phase-2.bib}

\newpage
\appendix
\section{Experimental details}\label{app:experiment-details}

The code for the experiments can be found at \url{https://github.com/ohinder/Lagrange-multipliers-behavior}.

\subsection{Solvers}

\paragraph{One-phase solver.} For the well-behaved interior point solver, given a problem of the form
\begin{flalign*}
\minimize &f(x) \\
\subjectTo &c(x) = 0 \\
&x_{L} \le x \le x_{U},
\end{flalign*}
we can re-write the constraints as 
\begin{flalign*}
\minimize &f(x) \\
\subjectTo &c(x) \le 0 \\
-&c(x) \le 0 \\
&x_{L} \le x \le x_{U}.
\end{flalign*}
This gives a problem of the form
\begin{flalign*}
\minimize &f(x) \\
\subjectTo &a(x)  + s = 0 \\
& s \ge 0,
\end{flalign*}
which we can pass to the one-phase solver.

The terms in Figure~\ref{fig:netlib-example} and Table~\ref{nonlinear-table} are given as follows:
\begin{itemize}
\item The infinity norm of the primal residual is given by $\| a(x) + s \|_{\infty}$. 
\item The infinity norm of the dual residual is measured by $\| \grad \Lag(x,y) \|_{\infty}$.
\item The infinity norm of complementarity is given by $\max_i s_i y_i$.
\item We measure strict complementarity by $\min_i s_i + y_i$.
\end{itemize}
The optimality termination criterion of the one-phase IPM is
$$
\max \left\{ \frac{100}{\max\{ \| y \|_{\infty}, 100 \}} \max\{ \| \grad_{x} \Lag(x,y) \|_{\infty}, \| Sy \|_{\infty} \},  \| a(x) + s \|_{\infty} \right\} \le 10^{-6}.
$$
For more details on the one-phase IPM see the paper \cite{hinder2018one} and code (\url{https://github.com/ohinder/OnePhase.jl}). The linear solver used was the default Julia Cholesky factorization (SuiteSparse).

\paragraph{IPOPT.} We use IPOPT 3.12.4 with the linear solver MUMPS. Given any generic nonlinear problem, IPOPT rewrites it in the form (by adding slacks to inequalities, see \cite{wachter2006implementation})
\begin{flalign*}
\minimize &f(x) \\
\subjectTo &c(x) = 0 \\
&x_{L} \le x \le x_{U}.
\end{flalign*}
For practical reasons related to the interface we use \cite{dunning2017jump}, we do this reformulation ourselves.
We then measure 
\begin{itemize}
\item Primal feasibility by $\| c(x) \|_{\infty}$.
\item Dual feasibility by $\| \grad f(x) + \grad c(x)^T \lambda - z_{L} + z_{U} \|_{\infty}$, where $z_{L}$ and $z_{U}$ are the dual multipliers corresponding to the constraint $x \ge l$ and $x \le u$ respectively (same notation as in \cite{wachter2006implementation}).
\item Complementarity is given by $\max\{ \max_i((z_{L})_i (x_i - l_i)), \max_i( (z_{U})_i (x_i - u_i)) \}$.
\item We measure strict complementarity by $\min\{ \min_i((z_{L})_i (x_i - l_i)), \min_i((z_{U})_i (x_i - l_i)) \}$.
\end{itemize}

The details of this computation can be found in the file `src/shared.jl' in the function `add\_solver\_results!'.

The options chosen for the solvers are given in Table~\ref{tbl:solver-options}. We turn off the acceptable termination criterion for IPOPT to try to make the termination criterion of the algorithms as similar as possible.

\begin{table}[H]
\begin{tabular}{c c}
\bottomrule
\multicolumn{2}{c}{IPOPT} \\
option & value\\
 \hline
max iter & \maxIter \\
 tol & $ 1.0 \times 10^{-6} $\\
 acceptable\_tol & $ 1.0 \times 10^{-6} $\\
 acceptable\_iter & 99999\\
 acceptable\_compl\_inf\_tol & $ 1.0 \times 10^{-6} $\\
 acceptable\_constr\_viol\_tol & $ 1.0 \times 10^{-6} $\\
  acceptable\_constr\_viol\_tol & $ 1.0 \times 10^{-6} $\\
 bound\_relax\_factor & $0.0$* \\
\bottomrule
\multicolumn{2}{c}{One Phase} \\
option & value\\
 \hline
max\_it & \maxIter \\
tol & $10^{-6}$ \\
\bottomrule
\end{tabular}
\caption{Solver options. The * indicates this option was only changed for `IPOPT w/o perturb'. For `IPOPT w. perturb' this was kept at its default value of $10^{-8}$.}\label{tbl:solver-options}
\end{table}

\subsection{NETLIB LP test details}\label{sec:netlib-details}

The linear programs in the NETLIB linear programming collection come in the form $\min{c^T x}$ s.t. $Ax = b$, $l \le x \le u$. Table~\ref{tbl:solver-failed} shows which solver failed on which problem.

Table~\ref{tbl:strict-interior} shows when there is a feasible solution according to Gurobi when the bound constraints are tightened by $\delta$ i.e. find a solution to the system $Ax = b$ and $u - \delta \ge x  \ge l + \delta$. We tried $\delta = 10^{-4}, 10^{-6}, 10^{-8}$ and obtained the same results with Gurobi's feasibility tolerance set to $10^{-9}$. We found $29$ problems with a feasible solution and $64$ without a feasible solution in the NETLIB collection.
  We used Gurobi version 7.02. 
  
  \begin{table}[H]
  \begin{tabular}{c p{7cm}}
\bottomrule
IPOPT w/o perturb & PEROLD, FFFFF800, SCAGR25, SHELL, SHARE1B, AGG3, VTP-BASE  (7 total) \\
\bottomrule
IPOPT w. perturb & PEROLD, FFFFF800, SCAGR25, SHELL, SHARE1B, VTP-BASE (6 total) \\ 
\bottomrule
One Phase & PEROLD, PILOT4, AGG2, PILOT-WE, GROW15, GROW22 (6 total) \\
\bottomrule
\end{tabular}
  \caption{NETLIB problems where a solver failed}\label{tbl:solver-failed}
 \end{table}
 
 \newpage
  
\begin{table}[H]
  \begin{tabular}{l l l l}
    \hline
    Problem name & strict interior & Problem name & strict interior  \\
\hline
    25FV47 & true & PILOT-JA & false \\
    80BAU3B & false & PILOT-WE & false \\
    ADLITTLE & false & PILOT & false \\
    AFIRO & true & PILOT4 & false \\
    AGG & false & PILOTNOV & false \\
    AGG2 & false & QAP12 & true \\
    AGG3 & false & QAP8 & true \\
    BANDM & false & RECIPELP & false \\
    BEACONFD & false & SC105 & false \\
    BLEND & true & SC205 & false \\
    BNL1 & false & SC50A & false \\
    BNL2 & false & SC50B & false \\
    BOEING1 & false & SCAGR25 & true \\
    BOEING2 & false & SCAGR7 & true \\
    BORE3D & false & SCFXM1 & false \\
    BRANDY & false & SCFXM2 & false \\
    CAPRI & false & SCFXM3 & false \\
    CYCLE & false & SCORPION & false \\
    CZPROB & false & SCRS8 & false \\
    D2Q06C & false & SCSD1 & true \\
    D6CUBE & true & SCSD6 & true \\
    DEGEN2 & false & SCSD8 & true \\
    DEGEN3 & false & SCTAP1 & true \\
    DFL001 & false & SCTAP2 & true \\
    E226 & false & SCTAP3 & true \\
    ETAMACRO & false & SEBA & false \\
    FFFFF800 & false & SHARE1B & true \\
    FINNIS & false & SHARE2B & true \\
    FIT1D & true & SHELL & false \\
    FIT1P & true & SHIP04L & false \\
    FIT2P & true & SHIP04S & false \\
    FORPLAN & false & SHIP08L & false \\
    GANGES & false & SHIP08S & false \\
    GFRD-PNC & false & SHIP12L & false \\
    GREENBEA & false & SHIP12S & false \\
    GREENBEB & false & SIERRA & false \\
    GROW15 & true & STAIR & false \\
    GROW22 & true & STANDATA & false \\
    GROW7 & true & STANDGUB & false \\
    ISRAEL & true & STANDMPS & false \\
    KB2 & true & STOCFOR1 & true \\
    LOTFI & true & STOCFOR2 & true \\
    MAROS & false & TRUSS & true \\
    MODSZK1 & false & VTP-BASE & false \\
    NESM & false & WOOD1P & false \\
    PEROLD & false & WOODW & false \\
    QAP15 & true & & \\
    \hline
  \end{tabular}
    \caption{Problems in NETLIB collection with a strict relative interior}\label{tbl:strict-interior}
\end{table}

\subsection{Additional figures for nonconvex problems}\label{sec:nonconvex-details}

This section gives plots of solver trajectories for the nonconvex problems of Section~\ref{sec:nonconvex-programs}.

\begin{figure}
\centering
\includegraphics[scale=0.52]{\figuresDir{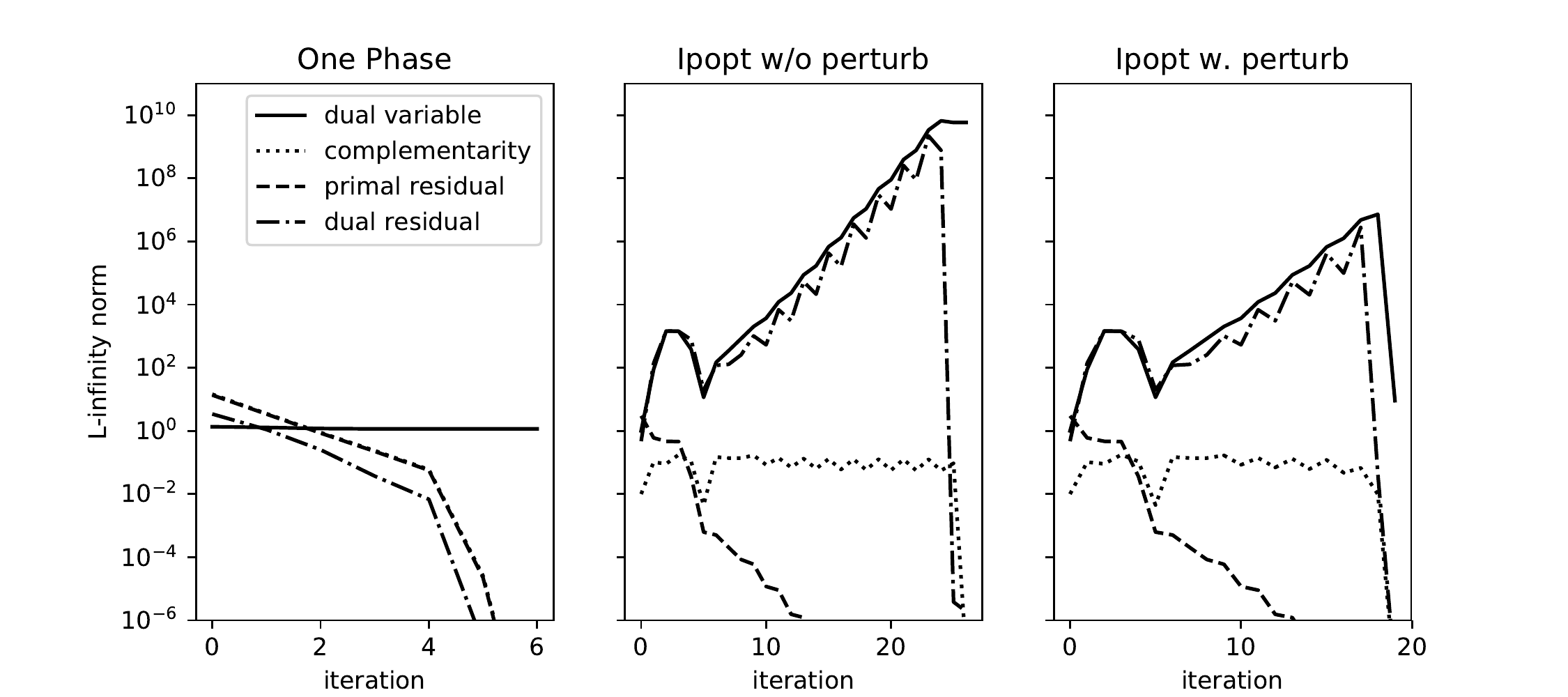}}
\caption{Comparison on the problem of finding the intersection of two circles.}\label{fig:example-circle}
\includegraphics[scale=0.52]{\figuresDir{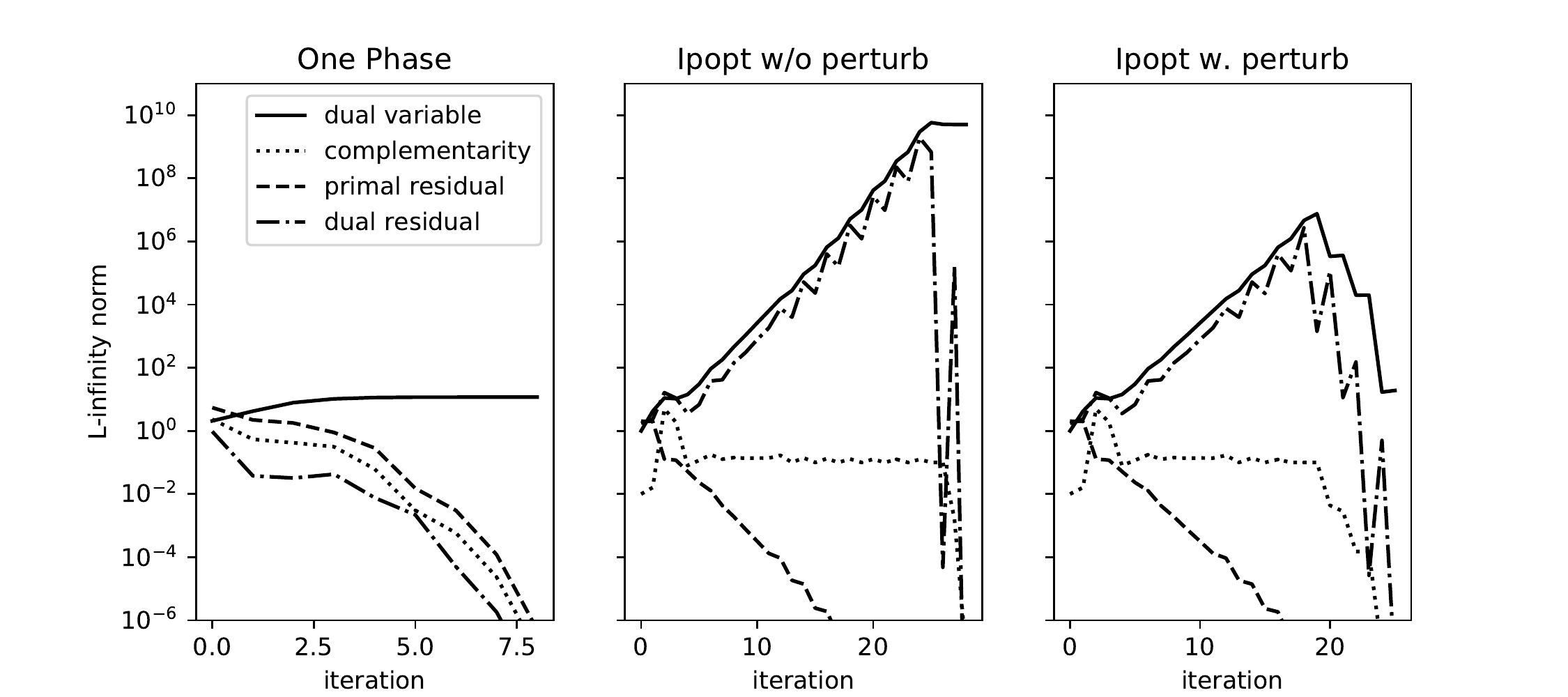}}
\caption{Comparison on a linear program with complementarity constraints.}\label{fig:example-comp}
\includegraphics[scale=0.52]{\figuresDir{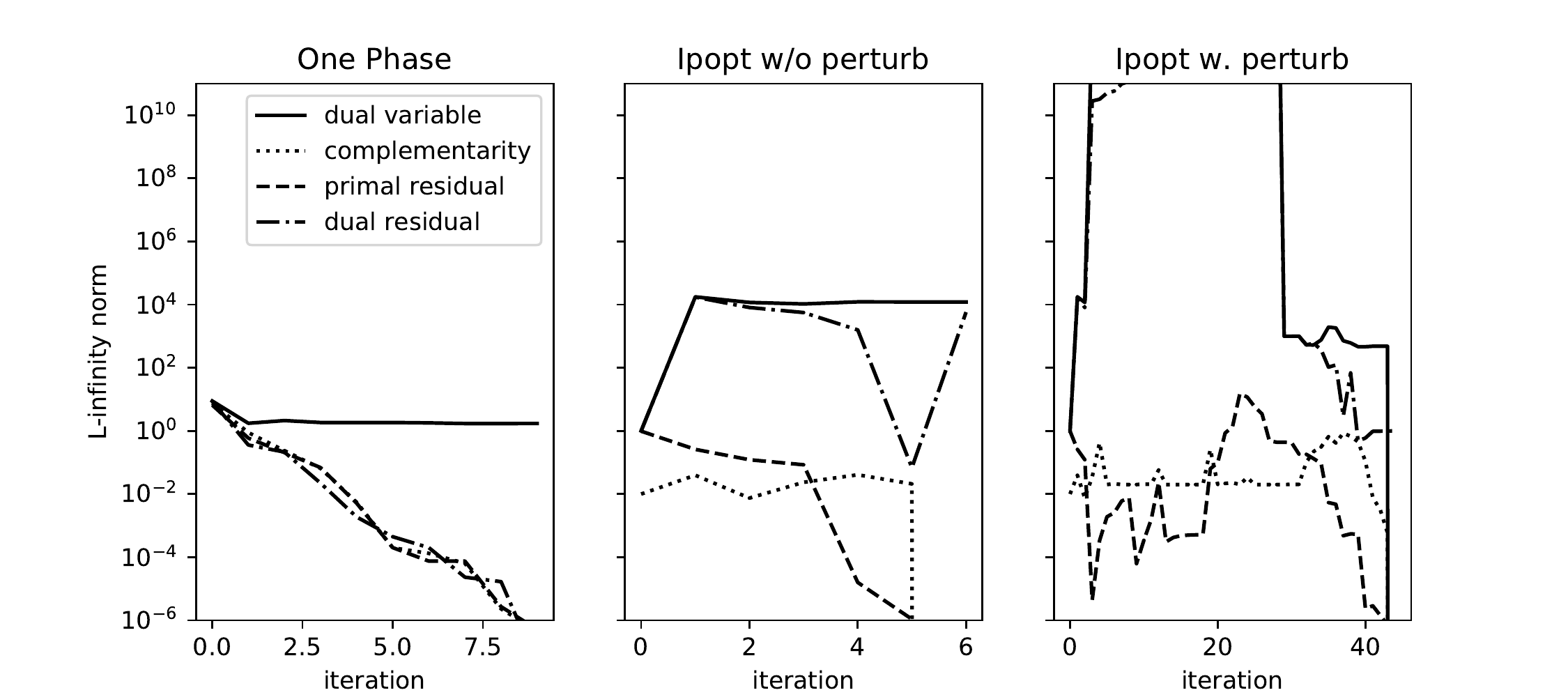}}
\caption{Comparison on a toy drinking water network optimization problem.}\label{fig:example-drink}
\end{figure}

%
%


\end{document}